\documentclass[12pt,dvips,reqno]{amsart}
\usepackage{euler, amsfonts, amssymb, latexsym, epsfig,epic,stmaryrd}
\usepackage{fancybox}

\newtheorem{Theorem}{Theorem} 
\newtheorem{Proposition}{Proposition} 
\newtheorem{Lemma}{Lemma}

\newtheorem{Corollary}{Corollary}
\newtheorem*{Corollary*}{Corollary}
 
\newtheorem*{Theorem*}{Theorem}
\theoremstyle{remark}
\newtheorem{Example}{Example}

\newcommand\lie[1]{{\mathfrak #1}}

\newcommand\diag{{\rm diag}}
\newcommand\Proj{{\rm Proj\,}}
\newcommand\iso{{\mathrel{\ \cong\ }}}
\newcommand\tensor{{\otimes}}

\newcommand\calO{{\mathcal O}}

\newcommand\onto{\mathop{\twoheadrightarrow}}

\newcommand\into{\operatorname*{\hookrightarrow}}
\newcommand\union{\cup}
\newcommand\Union{\bigcup}

\newcommand\Pone{{\mathbb P}^1}

\newcommand\CP{{\mathbb C \mathbb P}}

\newcommand\reals{{\mathbb R}}
\newcommand\complexes{{\mathbb C}}
\newcommand\integers{{\mathbb Z}}
\newcommand\rationals{{\mathbb Q}}

\newcommand\naturals{{\mathbb N}}

\newcommand\nulset{\emptyset}

\theoremstyle{plain}

\renewenvironment{quotation}
{\list{}{%\setlength\listparindent{0.5em}%
    \setlength\itemindent{0em}%
    \setlength\leftmargin{1.5em}
    \setlength\rightmargin{1.5em}
  }%
\item[]}
{\endlist}

\newcommand\dfn{\bf} % maybe should be \em

\newcommand\Gm{{\mathbb G}_m}

\newcommand\kk{{\mathbb F}}

\newcommand\Sym{\mathrm{Sym}}

\newcommand\junk[1]{}

\newcommand\PP{{\mathbb P}}
\renewcommand\AA{{\mathbb A}}

\begin{document}
\pagestyle{plain}

\title{A compactly supported formula for equivariant localization \\
  \lowercase{and} \\
  simplicial complexes of Bia\l ynicki-Birula decompositions}

\author{Allen Knutson}
\thanks{Supported by an NSF grant.}
\email{allenk@math.ucsd.edu}
\dedicatory{Dedicated to Sir Michael Atiyah and to Isadore Singer
  on their 80th and 85th birthdays}
\date{January 2008}

\maketitle

\begin{abstract}
  The Duistermaat-Heckman formula for their induced measure on a
  moment polytope is nowadays seen as the Fourier transform of the
  Atiyah-Bott localization formula, applied to the $T$-equivariant
  Liouville class.  From this formula one does not see directly that
  the measure is positive, nor that it vanishes outside the moment polytope.
  
  In [Knutson99] we gave a formula for the Duistermaat-Heckman measure
  whose terms are all positive and compactly supported, using a
  Morse decomposition. Its derivation required that the stable and
  unstable Morse strata intersect transversely.
  
  In this paper, we remove this very restrictive condition, at the
  cost of working with an ``iterated'' Morse (or Bia\l ynicki-Birula)
  decomposition. This leads in a natural way to a simplicial complex
  of ``closure chains'', which in the toric variety case is just a
  pulling triangulation of the moment polytope.
  To handle the singularities of the closed strata we restrict to the
  projective algebraic setting.  Conversely, this allows us to work
  from the beginning with singular projective schemes over
  algebraically closed ground fields.
\end{abstract}

{\small\small
\tableofcontents}

%\numberwithin{Theorem}{section}
\numberwithin{Lemma}{section}
\numberwithin{Proposition}{section}
\numberwithin{Corollary}{section}

\section{Background, and statement of results}

Let $X \subseteq \PP V$ be a projective algebraic variety over
an algebraically closed field, invariant under the linear action 
of a torus $T$ on $V$. Then (as in \cite{Brion})
there is an associated {\dfn Duistermaat-Heckman measure} $DH(X,T)$
on the dual $\lie{t}^*$ of the Lie algebra,
the weak limit as $n\to\infty$ of the Dirac measures
$$ \sum_{\mu \in T^* \subseteq \lie{t}^*} 
\frac{\dim\left( \text{$\mu$-weight space in }\Gamma(X; \calO(n)) \right)}
{n^{\dim X - \dim T}}
\ \delta_{\mu/n}. $$
As shown in \cite{Brion},
this measure $DH(X,T)$ is supported on the convex hull of the weights
of $T$ acting on the lines $\calO(1)|_{x \in X^T}$ over the fixed points,
and is a piecewise-polynomial times Lebesgue measure on that polytope,
called the {\dfn moment polytope}. 
It is a pleasant way to encode the asymptotics of the $T$-representation
$\Gamma(X; \calO(n)), n\to\infty$.

In the simplest case, $T=1$, this gives a Dirac measure times the
leading coefficient $\deg X / (\dim X)!$ of the Hilbert polynomial.
More generally, the value of this function at a interior integral point $p$
of this polytope is the leading coefficient of the Hilbert polynomial 
of the geometric invariant theory quotient $X //_p T$, 
with the linearization on $\calO(1)$ twisted by the character $-p$.\footnote{%
  The quotient $X //_p T$ may only carry a $1$-dimensional sheaf, rather
  than a line bundle, but this does not affect the definitions in
  any appreciable way.}
(One can also extend this definition to rational $p$, and we will
state a more general result of this type in proposition \ref{prop:D-Hgeneral}.)

If $T' \to T$ is a homomorphism (e.g. the inclusion of a subtorus),
then there is a natural map $\lie{t}^* \to \lie{t}'^*$ taking 
$DH(X,T)$ to $DH(X,T')$. 
For example, the $T'=1$ case lets one compute the degree
using the total mass of the Duistermaat-Heckman measure.

The polytope and measure are named for their origins
in the case that the base field is $\complexes$ \cite{D-H}.
If one chooses a Hermitian metric on $V$ invariant under the compact
subgroup $T_\reals$ of $T$, then there is a {\em moment map} 
$\Phi_T : X \to \lie{t}^*$ whose image is exactly the moment polytope,
and $DH(X,T)$ is the pushforward along $\Phi_T$
of the Liouville measure on the (smooth part of the) variety $X$.
One property of this map $\Phi_T$ is that for $f\in X^T$, 
the value $\Phi_T(f)\in T^*$ 
is the $T$-weight on the line $\calO(1)|_f$; as such we will use
$\Phi_T(f)$ to denote this weight even when the base field is not $\complexes$
(though $\Phi_T(x)$ will not be defined for $x\notin X^T$).

Hereafter we assume the fixed point set $X^T$ is isolated. 
Under this assumption Duistermaat and Heckman gave a formula for their
measure as an alternating sum over $X^T$ (this version is from
\cite[proposition 3.3 and its preceding theorem]{GLS}):

\newcommand\sign{\text{sign}}

\begin{Theorem*}\cite{D-H,GLS}
  Let $X$ % \subseteq \PP V$ 
  be a compact symplectic manifold of dimension $2n$ 
  with symplectic form $\omega$ and Liouville measure $[\omega^n]$,
  and $T$-moment map $\Phi_T : X \to \lie{t}^*$. 

  For each fixed point $f\in X^T$, let
  $\lambda^f_1,\ldots,\lambda^f_n$ be the weights of $T$ acting
  on the tangent space $T_f X$.
  Pick $\vec v \in \lie{t}$ such that 
  $\langle \vec v, \lambda^f_i\rangle \neq 0$ for all $f\in X^T, i=1,\ldots n$.
  (In particular each $\lambda^f_i \neq 0$, which is the condition
  that $X^T$ is isolated.)
  
  Then the measure $DH(X,T) := (\Phi_T)_*(\frac{[\omega^n]}{n!})$ on 
  the moment polytope $\Phi_T(X)$ equals the sum
  $$ \sum_{f \in X^T} 
  \sign \left(\prod_{i=1}^n \langle \vec v, \lambda_f^i\rangle\right) (C_f)_*
  \left( \text{Lebesgue measure on the orthant }\reals_{\geq 0}^n \right) $$
  where $C_f : \reals^n \to \lie{t}^*$ is the affine-linear map
  $$ C_f : \reals^n \to \lie{t}^*, \qquad 
  (r_1,\ldots,r_n) \mapsto \Phi_T(f) + \sum_{i=1}^n 
  r_i\ \sign\left(\langle \vec v, \lambda_f^i\rangle\right) \lambda_f^i $$
  which is proper when restricted to $\reals_{\geq 0}^n$. 
\end{Theorem*}

In particular each term is supported on a noncompact polyhedral cone,
and much cancelation occurs to produce a compactly supported answer.

However, note that once one has computed $DH(X,T)$, one can use the restriction
$1 \into T$ to compute the symplectic volume $\int_X e^\omega$ of $X$.
There are two interesting subtleties in this restriction.
One is that we can't pass from $T$ to the trivial group and then
apply the theorem, because we lose the ``$X^T$ isolated'' condition.
The other is that the total mass of $DH(X,T)$
can't be computed term-by-term, since the mass of each term is infinite.
(This latter problem can be fixed rather crudely by cutting $\lie{t^*}$ 
with a half-space chosen to contain $\Phi_T(X)$, or even just to contain 
the point at which one wishes to evaluate $DH(X,T)$.)

We mention that one can see from the above formula (or more directly)
that if the kernel of $T$'s action on $X$ is finite, i.e. at some
(hence every) fixed point $f$ the $\{\lambda_f^i\}$ rationally span
$\lie{t}^*$, then $DH(X,T)$ is Lebesgue measure times a 
piecewise-polynomial function, called the {\dfn Duistermaat-Heckman function}.

\newcommand\tomega{\widetilde\omega}
In their very influential paper \cite{AB}, 
Atiyah and Bott (at the same time as Berline and Vergne in \cite{BV})
gave a formula for the integration of equivariant cohomology classes
on a compact $T$-manifold, and showed that the Duistermaat-Heckman formula 
is the special case of integrating the exponential of
the {\em equivariant} symplectic form $\tomega := \omega - \Phi_T$.

\begin{Theorem*}\cite{AB,BV}
  Let $X$ be a compact oriented manifold, and $\alpha \in H^*_T(X)$,
  where $T$ acts on $X$ with isolated fixed points $X^T$, and as above
  let $\lambda^f_1,\ldots,\lambda^f_n$ be the weights of $T$ acting
  on the tangent space $T_f X$ for each fixed point $f\in X^T$.

  Then the pushforward of $\alpha$ along the map $X\to pt$, denoted
  $\int_X \alpha \in H^*_T(pt) \iso \Sym(T^*)$, can be computed as
  $$ \int_X \alpha 
  = \sum_{f\in X^T} \frac{\alpha|_f}{\prod_{i=1}^n \lambda_f^i} $$
  where the right-hand side formally lives in the ring of fractions
  of the polynomial ring $\Sym(T^*)$. 
  Here $\alpha|_f \in H^*_T(f) \iso H^*_T(pt)$ denotes
  the pullback of $\alpha$ along the $T$-equivariant inclusion 
  $\{f\} \into X$.
\end{Theorem*}

By definition, the equivariant cohomology ring $H^*_T(X)$ is the
direct sum of the groups $H^i_T(X)$. But the AB/BV formula obviously
extends to elements of the direct product $\prod_i H^i_T(X)$,
such as $\exp(\tomega)$. Then
$$ \int_X \exp(\tomega) = \sum_{x\in X^T} 
\frac{\exp\left(-\Phi_T(x)\right)}{\prod_{i=1}^n \lambda_x^i}. 
$$
It is very tempting to Fourier transform term-by-term, turning 
$\exp(-\Phi_T(f))$ into $\delta_{\Phi_T(f)}$, and the division by $\lambda_f^i$
into integration in the $\lambda_f^i$ direction. Making proper sense of this
(fixing the constant of integration, one might say)
requires the choice of $\vec v$ from the Duistermaat-Heckman theorem, and 
flipping of those weights for which $\langle \vec v,\lambda_f^i \rangle < 0$.
That done, the Duistermaat-Heckman theorem (in the \cite{GLS}
form above) follows.
\subsection{The basic formula}

Hereafter we work in the algebro-geometric setting, largely to 
avoid questions relating to singularities of certain subsets of $X$; 
our localization theorem will thus
be for equivariant Chow classes (see e.g. \cite{eqvtChow}). 
On the plus side, we will not require any smoothness assumption on $X$ itself;
hereafter, throughout the paper, $X$ will always denote a projective scheme
(except for a brief discussion in section \ref{ssec:myDHgen}).
If the components of $X$ are of varying dimension, $\dim X$ means 
the maximum thereof. 

Our main result is a different formula for the Duistermaat-Heckman measure,
in which all the terms are themselves positive,
and perforce compactly supported. 
Rather than maps of the orthant to $\lie{t}^*$, 
the terms will be based on maps of the {\dfn standard $n$-simplex}
$\{ \vec v \in \reals_{\geq 0}^n : \sum_i v_i \leq 1 \}$ to $\lie{t}^*$.

Fix a one-parameter subgroup\footnote{%
  While this $S$, or rather its associated coweight,
  bears superficial similarity to the vector $\vec v \in \lie{t}$ needed in
  the Duistermaat-Heckman theorem, its usage will be substantially different.}
$S: \Gm \to T$ such that $X^S = X^T$. 
Then the {\dfn Bia\l ynicki-Birula stratum} \cite{BB},
hereafter {\dfn B-B stratum}, $X_f$ is defined as the locally closed subset
$$ X_f := \{x \in X : \lim_{z\to 0} S(z)\cdot x = f \} $$
(considered with the reduced scheme structure, i.e., as a set).
It is easy to see that $X = \coprod_{f\in X^T} X_f$; this is called
the {\dfn B-B decomposition} of $X$ (or more precisely, the pair $(X,S)$), 
and is the algebraic analogue of a Morse decomposition. 

Unfortunately, this is usually not a stratification: 
the closure $\overline{X_f}$ is usually not a union of strata $\{X_g\}$
(one example to be given in section \ref{ssec:TVs}).
Consequently, the combinatorics of the finite set $X^T$ is much
richer than just a partially ordered set 
(though it is\footnote{%
  This statement is quite nonobvious, actually, as it uses
  projectivity in a crucial way: otherwise one can glue two $\Pone$s
  together, each one carrying the standard action of $\Gm$,
  but each one's $\vec 0$ glued to the other one's $\vec\infty$.
  Even smooth counterexamples have been constructed \cite{Ju}.
  In the smooth projective case, this statement appears in \cite{BB},
  and more generally can be proven with the technique of 
  lemma \ref{lem:hesselink}, though we will never use it directly.}
that, by taking the
transitive closure of ``$g\geq f$ if $g\in \overline{X_f}$'').
Define $\overline{X_{f_0,\ldots,f_k}}$ inductively by
$$ \overline{X_\nulset} := X, \qquad
\overline{X_{f_0,\ldots,f_k}} := \overline{
\overline{X_{f_0,\ldots,f_{k-1}}} \cap X_{f_k}}. $$
Call a nonrepeating sequence $\gamma = (f_0,\ldots,f_{k-1})$ 
a {\dfn closure chain} if $\overline{X_{f_0,\ldots,f_k}}$ is nonempty, 
or equivalently, if $\overline{X_{f_0,\ldots,f_k}} \ni f_k$.
Obviously this implies $f_0 < f_1 < \ldots < f_k$, hence one can %safely 
think of $\gamma$ as just a set,
with the partial order on $X^T$ ``remembering'' the order on $\gamma$.

It is easy to see that the set of closure chains forms a simplicial complex
$\Delta(X,S)$ (meaning, any subset of a closure chain is itself one). Note that
$\overline{X_{f_0,\ldots,f_{k-1}}} \cap X_{f_k}$ is itself a B-B stratum, 
namely $\left(\overline{X_{f_0,\ldots,f_{k-1}}}\right)_{f_k}$
in $\overline{X_{f_0,\ldots,f_{k-1}}}$, and hence connected when nonempty.
We christen the set of all these subsets
$\left\{ \left(\overline{X_{f_0,\ldots,f_{k-1}}}\right)_{f_k} \right\}$
the {\dfn iterated B-B filtration of $(X,S)$}.
Our most nontrivial result (proposition \ref{prop:equidim})
about the complex $\Delta(X,S)$ is that it is equidimensional when $X$ is.

At this point, we can give a weak statement of our version of the
Duistermaat-Heckman formula. It will be in terms of the simplicial
complex $\Delta(X,S)$, which does not depend on the projective
embedding, and some coefficients $\{v_\gamma \in \naturals\}$ that do. 
We defer a precise definition of these coefficients until theorem
\ref{thm:vrecurrence}, and until then this is a sort of existence result.

\begin{Theorem}\label{thm:myDH}
  Let $X \subseteq \PP V$ be a subscheme invariant under the
  linear action of a torus $T$ on the vector space $V$.
  Assume that the fixed point set $X^T$ is finite, and 
  let $S : \Gm \to T$ be a one-parameter subgroup with $X^S = X^T$,
  with which to define the complex $\Delta(X,S)$ of closure chains.

  The longest a closure chain $\gamma$ may be is $1 + \dim X$ elements.
  (If $X$ is equidimensional, then every maximal closure chain is
  indeed this long.)
  To each such closure chain $\gamma$, and depending on the projective
  embedding, there is associated a positive integer $v_\gamma$, 
  such that the Duistermaat-Heckman measure of $X$ can be calculated as
  $$ DH(X,T) = \sum_\gamma v_\gamma\ (C_\gamma)_*
  \left( \text{Lebesgue measure on the standard $n$-simplex} \right) $$
  where $C_\gamma$ is the unique affine-linear map $\reals^{n} \to \lie{t}^*$
  taking the vertices of the standard simplex to 
  $\{ \Phi_T(f) : f \in \gamma\}$.

  In particular, to determine the value at a point $p$, we need only
  sum over those $\gamma$ such that $p$ lies in the convex hull of
  $\{ \Phi_T(f) : f \in \gamma\}$.
\end{Theorem}

Note that the Duistermaat-Heckman measure is not
sensitive to components of lower dimension (geometric or embedded),
so one may freely replace $X$ by the union of its primary components
of top dimension. It matters little because though this replacement may 
shrink $\Delta(X,S)$, it doesn't change the set of faces $\gamma$ summed over
(as follows from corollary \ref{cor:deltaunion} and 
proposition \ref{prop:equidim}).

As we explained after the Duistermaat-Heckman theorem, it is very
tricky to turn their formula into one for the symplectic volume 
(or in the algebraic situation, the degree).
Whereas here, since the individual terms have finite volume, 
we {\em can} forget the $T$-action term by term and obtain the formula
$\deg(X) = \sum_\gamma v_\gamma$.
We refine this sum in section \ref{ssec:assembling}.

%The definition of the coefficients $v_\gamma$ is unfortunately 
%rather subtle.  We include some discussion of it in section \ref{sec:proofs}.

\newcommand\codim{\mathop {\rm codim}}

\subsection{Examples of $\Delta(X,S)$}

\subsubsection{Flag manifolds}

In the case that the B-B decomposition is a stratification,
then each nonempty $\left(\overline{X_{f_0,\ldots,f_{k-1}}}\right)_{f_k}$
is just $X_{f_k}$, and any chain $f_0 < \ldots < f_k$ in the 
partial order is a closure chain. So the complex $\Delta(X,S)$ 
of closure chains is just the ``order complex'' of this poset $X^T$. 
Under a slightly stronger assumption, the theorem \ref{thm:myDH} here
is an algebro-geometric version of theorem 1 in our earlier paper 
\cite{LittDH}, 
proven there for symplectic manifolds with no algebraicity condition.

Our inspiration for that formula was the case $X = G/P$ 
a generalized flag manifold, where the B-B decomposition is the 
Bruhat decomposition \cite{Akyildiz} and the partial order is the Bruhat order.
In this case the order complex is homeomorphic to a ball \cite{BjWa}.

Each space $\Gamma(G/P; \calO(n))$ is an irreducible representation of $G$, 
so the exact formula for the $T$-equivariant Hilbert function (not just its 
asymptotics) is given by the Kostant multiplicity formula or the Littelmann 
path formula, one case of which was the Lakshmibai-Seshadri conjecture.
As explained in \cite{LittDH}, the asymptotics of the Kostant and
Lakshmibai-Seshadri formulae reproduce respectively the Duistermaat-Heckman 
theorem (this special case being Heckman's thesis) or theorem \ref{thm:myDH}.

All the same analysis goes over to Schubert varieties inside flag manifolds,
not just the flag manifolds themselves. The resulting formula for degrees of
Schubert varieties is closely related to the one in \cite{PS},
and more distantly to the one in \cite{Duan}.

\subsubsection{Toric varieties}\label{ssec:TVs}

Let $X$ be the complex toric variety associated to an integral polytope 
$P \subseteq \lie{t}^*$. Each of the subsets $\overline{X_{f_0,\ldots,f_k}},
\left(\overline{X_{f_0,\ldots,f_{k-1}}}\right)_{f_k}$
in the iterated B-B filtration
of $X$ maps under the moment map $\Phi_T$ onto a corresponding subset 
$\overline{P_{f_0,\ldots,f_k}},
\left(\overline{P_{f_0,\ldots,f_{k-1}}}\right)_{f_k}$ of $P$,
which will be easier to visualize.

The choice $S: \Gm \into T$ defines an ``up'' direction on $P$; 
the condition $X^T = X^S$ says that each edge (hence each face) 
has a top vertex and a bottom vertex. 
Then $P_f$ (resp. $\overline{P_f}$) 
is the union of those open faces (resp. closed faces) 
of $P$ whose bottom vertex is $f$.
If $P$ is a simple polytope, meaning that there are 
only $\dim P$ edges from each vertex (equivalently, $X$ has at
worst orbifold singularities),
then $P_f$ contains only one maximal face, but this is not always true:
consider $P$ an octahedron almost balanced on one corner,
tilted over a little. Then the lowest
of the four points on the equator has $\overline{P_f} = $ two triangles.

In this case $\Delta(X,S)$ is a well-known triangulation of $P$ 
(a ``pulling triangulation'' by pulling the vertices starting from the bottom).
More precisely,
the maps $C_\gamma$ from the standard simplex to $\lie{t}^*$ are embeddings,
and their images in $\lie{t}^*$ exactly cover $P$.

We illustrate in the case $P$ a truncated right triangle, so $X$ a
Hirzebruch surface $F_1$, where one can already see the B-B decomposition
fail to be a stratification \cite[example 1]{BB}. 
Pictured left-to-right are $P$, its B-B decomposition, and $\Delta(X_P,S)$.

\centerline{\epsfig{file=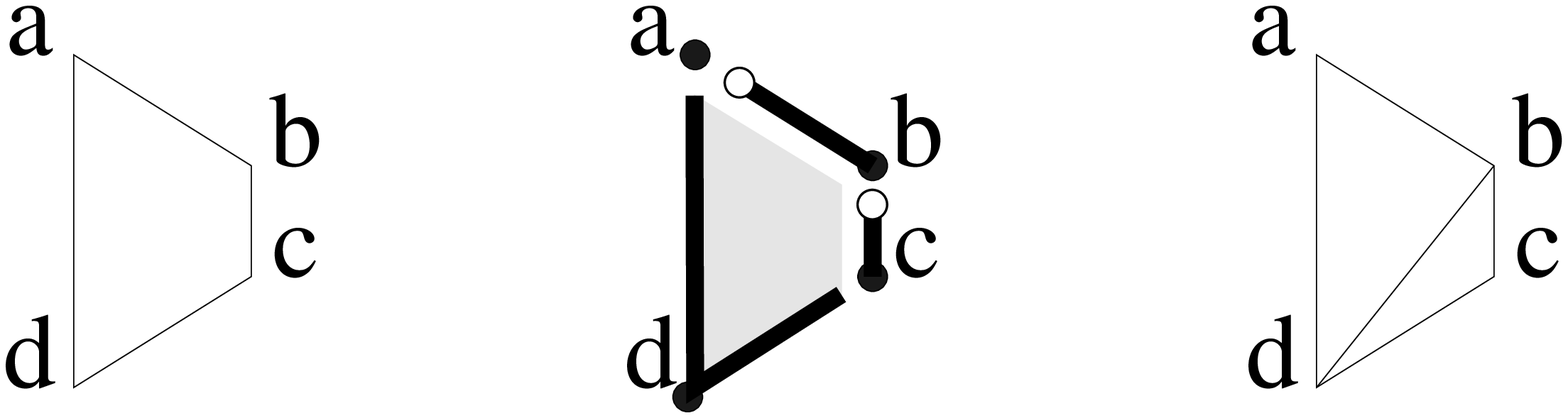,width=3in}}

\noindent 
The closure of $X_c$ is not a union of strata (it intersects but
doesn't contain $X_b$)
and even though $b \in \overline{X_c}, a \in \overline{X_b}$, 
we don't have $a \in \overline{X_c}$ nor a triangle in $\Delta(X_P,S)$ 
containing $\{a,b,c\}$.
The coefficient on $v_\gamma$ is the volume of the convex hull of 
$\Phi_T(\gamma) \subseteq P$, 
which partly motivated the choice of the letter $v$.

\subsubsection{Stanley-Reisner schemes}
In both the flag manifold and toric variety examples, the simplicial
complexes $\Delta(X,S)$ were very special: they were homeomorphic to balls.
In this section we show that in the general case, {\em any}
simplicial complex may arise.
The following examples do not provide interesting applications of
theorem \ref{thm:myDH}, but are good for testing one's intuition
about closure chains.

Let $\Delta$ be an arbitrary simplicial complex on the vertex set $1,\ldots,n$,
and let $V = \AA^n$. To each face $F$ in $\Delta$, we associate the 
coordinate projective subspace $X(F) \subseteq \PP V$ 
that uses only those coordinates.
Then let $X(\Delta) = \cup_{F\in\Delta} X(F)$ be the union
of those coordinate projective subspaces. 
This $X(\Delta)$ is the {\dfn (projective) Stanley-Reisner scheme} of $\Delta$,
whose coordinate ring is the (homogeneous) Stanley-Reisner ring of $\Delta$.
It is invariant under the torus $T$ that scales each coordinate 
independently.

Let $\Phi_S : \{1,\ldots,n\} \into \integers$ be strictly increasing.
(Really the important condition is injectivity, but by permuting
$1,\ldots,n$ we can obtain this convenient stronger condition.)
Then there is a corresponding action of $\Gm$ on $\PP V$, by
$$ S(z) \cdot [x_1,\ldots,x_n] 
:= \left[ z^{\Phi_S(1)} x_1, 
\ldots, z^{\Phi_S(i)} x_i, \ldots, z^{\Phi_S(n)} x_n\right] $$
which fixes $X(\Delta)$. The condition that $\Phi_S$ is injective
says that the only $S$-fixed points on $\PP V$ are the coordinate points
(actually it is enough that $\Phi_S(f) \neq \Phi_S(g)$ for each 
edge $\{f,g\} \in \Delta$).

\begin{Proposition}\label{prop:SRcase}
  Let $\Delta,X(\Delta),\Phi_S,S$ be as above.
  Then the associated simplicial complex of closure chains is just $\Delta$.
  In particular, every finite simplicial complex arises in this way.
\end{Proposition}

\begin{proof}
  Identify the fixed points $\big\{ [0,\ldots,0,1,0,\ldots,0] \big\}$ 
  with $1,\ldots,n$. Then it is easy to show that
  $$  X(\Delta)_i 
  = \overline{X(\Delta)_i} \cap \big\{ [x_1,\ldots,x_n] : x_i \neq 0 \big\}
  \qquad\text{where}\qquad
  \overline{X(\Delta)_i} = \Union_{F \in \Delta,\ \min(F) = i} X(F). $$
  (This would be more irritating to state without having
  first made $\Phi_S$ strictly increasing.)  
  From this, one can show inductively that
  $$ \overline{X(\Delta)_{f_0,\ldots,f_k}} = \Union 
  \big\{ X(F)\ :\ F \in \Delta,\ 
  F = \{f_0,\ldots,f_k,\text{ larger numbers}\} \big\}$$
  so the left side is nonempty iff $\{f_0,\ldots,f_k\}$ is the initial
  string of a face of $\Delta$, i.e. iff it is a face of $\Delta$.
\end{proof}

When theorem \ref{thm:myDH} is applied to $X(\Delta)$, each 
mysterious coefficient $v_F$ is just $1$, 
the degree of the projective variety $X(F)$.

\subsubsection{Some tricky behavior}\label{sssec:tricky}
We first mention a geometric subtlety of the definition of closure chain.
Plainly $\overline{X_{f_0,\ldots,f_k}}$ is contained in 
$\overline{X_{f_0,\ldots,f_{k-1}}} \cap \overline{X_{f_k}}$, 
since it is defined as the closure of
$\overline{X_{f_0,\ldots,f_{k-1}}} \cap X_{f_k}$.
But it can be strictly smaller, as we will show by example in a moment. 
One can show that if $\overline{X_{f_0,\ldots,f_k}}$ had instead been 
{\em defined} as $\overline{X_{f_0,\ldots,f_{k-1}}} \cap \overline{X_{f_k}}$,
then the (similarly larger) complex of closure chains would be a
``clique complex'', meaning, the largest simplicial complex with a
given set of $1$-faces.  For example, order complexes of posets are
clique complexes, where the $1$-faces $\{a,b\}$ specify comparability
of $a$ and $b$.

The smallest simplicial complex that isn't a clique complex is a
hollow triangle (the clique complex would be the solid triangle). 
The corresponding Stanley-Reisner scheme
is $X = \{[a,b,c] : abc = 0\} \subseteq \PP^2$.
Taking $S(z)\cdot [a,b,c] := [a,zb,z^2c]$, the B-B strata are
$$ X_{[1,0,0]} = \{[1,b,0]\} \cup \{[1,0,c]\},\quad
X_{[0,1,0]} = \{[0,1,c]\},\quad
X_{[0,0,1]} = \{[0,0,1]\} $$
and the complex of closure chains is the desired hollow triangle.
In this example, we see the claimed geometric subtlety
at $\overline{X_{[1,0,0],[0,1,0]}} = \{[0,1,0]\}$,
contrasted with the strictly larger
$\overline{X_{[1,0,0]}} \cap \overline{X_{[0,1,0]}} = \{[0,1,0],[0,0,1]\}$.

Another interesting (for other reasons) example in the plane is
$X = \{[a,b,c] : b(ac-b^2) = 0\}$, invariant under the same $S$ as above,
and we take $T=S$.
This has two $S$-fixed points $\{[1,0,0],[0,0,1]\}$,
and the complex of closure chains is an interval. 
(The unique $v_\gamma$ turns out to be $3$, the degree of $X$.)
Whereas the geometry of $X$ -- a line union a conic, meeting transversely at 
the two $S$-fixed points -- would seem to suggest that the more
appropriate complex would be an oval, made with two intervals glued
together at both ends 
(their $v_\gamma$s being $1$ for the line and $2$ for the conic).

However, that is not a simplicial {\em complex} (in which faces 
are determined by their set of vertices), but falls under the
slightly more general notion of simplicial {\em poset}. 
We did not need this richer notion to formulate theorem \ref{thm:myDH},
but we will use it in the $K$-theory version \cite{AKfuture},
based on ideas from \cite{balanced}.

Having just described a simplicial complex (the clique complex above) 
that is slightly larger than we need, the reader may wonder whether
the complex $\Delta(X,S)$ might still be larger than necessary.
One sign that the complex is a good one is that the coefficients $v_\gamma$
in theorem \ref{thm:myDH}
are strictly positive, so no term may be omitted. Another is that in
the toric variety case discussed in the previous section,
the supports of the terms are disjoint, so the $v_\gamma$ can't even
be adjusted to leave some term out.
Another indication of $\Delta(X,S)$'s minimality will come in
proposition \ref{prop:equidim}.

\subsubsection{A Bott-Samelson manifold}\label{sssec:BS}

\newcommand\Cthree{{\complexes^3}}
\newcommand\Gr{{\rm Gr}}
%\Ex{BS121}
{
Bott-Samelson manifolds provide examples of $\Delta(X,S)$ that are
{\em not} homeomorphic to balls, despite $X$ being irreducible. 
Consider the variety 
$X \subset \Gr_1(\Cthree)\times \Gr_2(\Cthree)\times \Gr_1(\Cthree)$
of triples of subspaces, with incidences specified by the Hasse diagram
$ X = \Bigg\{ (V_1,V_2,V_1') : 
\begin{array}{ccccccc}
\langle e_1, e_2 \rangle
%\{(0,\star,\star)\} 
& & & & V_2 & & \\
 | & \backslash & & / & & \backslash & \\
%\{(0,0,\star)\} 
\langle e_1 \rangle
& & V_1 & & & &  V_1'
\end{array}
\Bigg\} $
where $\{e_1,e_2,e_3\}$ denote the standard basis of $\complexes^3$.

This is the ``$1$-$2$-$1$ Bott-Samelson manifold''; % \cite{G1}; 
one way to think of it is as a walk from the base flag 
$(\langle e_1 \rangle, \langle e_1,e_2 \rangle)$
to other flags, by changing the line, then the plane, then the line again.
It carries an
action of the diagonal matrices $T$ inside $GL_3(\complexes)$ (indeed, of
the upper triangulars). There are $2^3 = 8$ $T$-fixed points, in which
$V_1,V_2,V_1'$ are coordinate subspaces. They are indexed by subsets
of the word $121$, where a letter is included if the corresponding 
subspace is different from the previous choice.

The Bott-Samelson is a blowup of the flag manifold, via the map
$(V_1,V_2,V_1') \mapsto (0 < V_1' < V_2 < \complexes^3)$. The exceptional
locus will turn out to be $\overline{X_{1--}}$. As such, $X$'s moment 
polytope is a subpolytope of that of the flag manifold, and we draw
it below:

\centerline{\epsfig{file=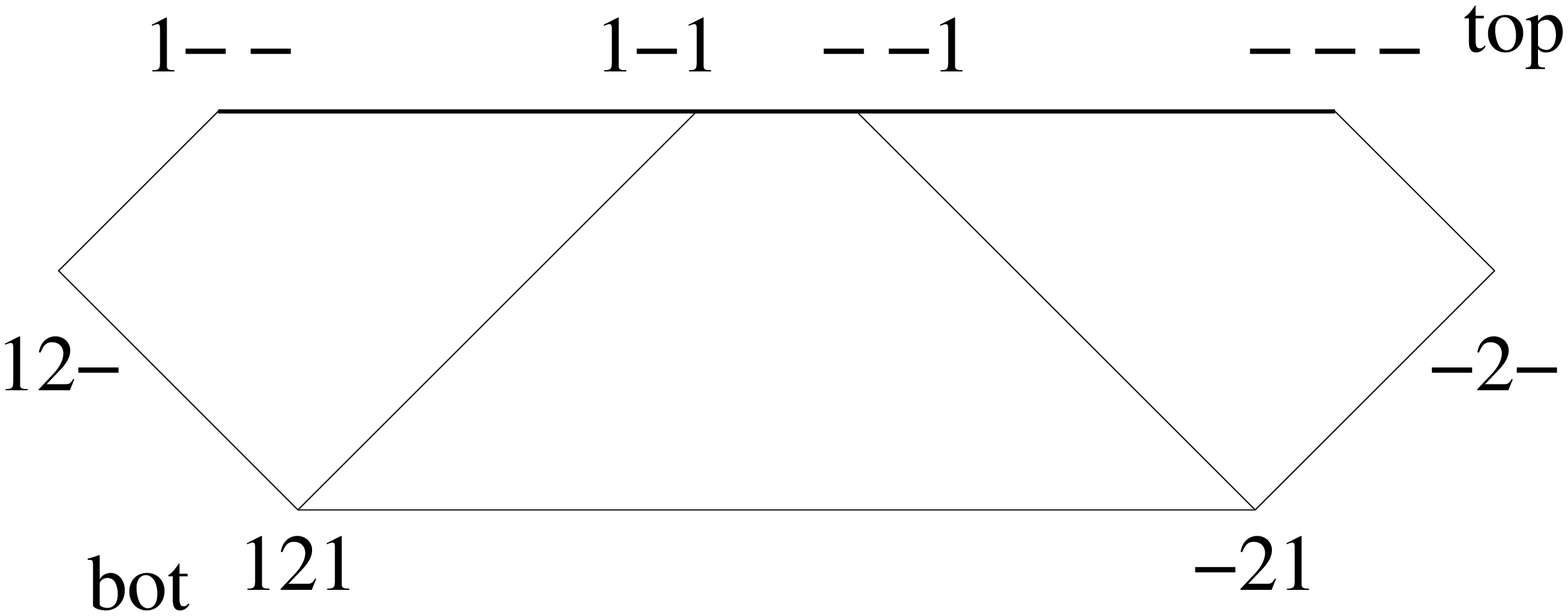, height=1.5in}}

Our one-parameter subgroup 
$S : \complexes^\times \to T$ will be $S(z) = \diag(1,z,z^2)$.
Then the top (most repellent) point is 
$(\langle e_1 \rangle, \langle e_1,e_2 \rangle, \langle e_1 \rangle)$,
and the bottom is
$(\langle e_2 \rangle, \langle e_2,e_3 \rangle, \langle e_3 \rangle)$.

The closures of the B-B strata are easy to compute:
$$ \overline{X_{121}} = X $$
$$ \overline{X_{-21}} = \{ V_1 = \langle e_1 \rangle \}, \quad
 \overline{X_{12-}} = \{ V_1 = V_1' \}, \quad
 \overline{X_{-2-}} = \{ V_1 = V_1' = \langle e_1 \rangle \} $$
$$ \overline{X_{1--}} = \{ V_2 = \langle e_1,e_2 \rangle \}, \quad
 \overline{X_{--1}} = \overline{X_{1--}} \cap \overline{X_{-21}},  \quad
 \overline{X_{1-1}} = \overline{X_{1--}} \cap \{ V_1' = \langle e_1 \rangle
\}$$
$$ \overline{X_{---}} = \{ V_1 = V_1' = \langle e_1 \rangle,
        V_2 = \langle e_1,e_2 \rangle \} $$

\newcommand\barX{{\overline X}}

To check that each $\barX_{f}$ is as claimed, note that it is
$T$-invariant, and has the right local behavior at $f$
(an easy tangent space calculation on $X_f$, which
is smooth because $X$ is smooth). So far this guarantees that the B-B stratum
$X_f$ is open inside the purported $\barX_{f}$. But then note that
each $\barX_{f}$ is irreducible, hence is the closure of $X_f$.

Note that this is {\em not} a 
stratification, as $\overline{X_{12-}} \not\supset \overline{X_{1--}}$.
Rather, 
$$ \overline{X_{12-,1--}} = \{V_1 = V_1' ,
        V_2 = \langle e_1,e_2 \rangle \}. $$ 
This is the diagonal of $\overline{X_{1--}} \iso (\CP^1)^2$, whereas
$\overline{X_{--1}}, \overline{X_{1-1}}$ are its two axes.

It remains to compute $\Delta(X,S)$. 
Because $\overline{X_{121}} = X$, the point $121$ will be a ``cone point'',
meaning that $\Delta(X,S)$ is a cone from that point.
Put another way, it is uninteresting, so let's leave it out for now.

Since the top point $\ ---\ $ is in every stratum closure, it also will be
a cone point. (While $bot$ being a cone point occurs whenever $X$ is
irreducible, $top$ being a cone point is much more a surprise; $top$
is not a cone point for most toric varieties, such as the Hirzebruch
surface in section \ref{ssec:TVs}.)

The complex $\Delta(X,S)$ is then the double cone (from $121$ and $---$) 
on the $1$-complex depicted below:

\junk{
\newcommand\barrr[1]{\begin{picture}(1, 1)
    \put(0,0){\line(1,0){#1}}  \end{picture} }
$$\begin{array}{ccccc}
\ovalbox{-21} &  \barrr{.3in}
%{\!\!\! \line(1,0){.3in}} \!\!\!}
 \quad   \ovalbox{-2-} \quad  \barrr{.3in}
%{\!\!\! \line(1,0){.3in}} \!\!\!}
& \ovalbox{12-} & & \\
 | &  &  | & & \\
\ovalbox{--1}&  \barrr{1.2in}
%{\!\!\! \line(1,0){1.2in}} \!\!\!}
%-\!\!-\!\!-\!\!-\!\!-\!\!-\!\!-\!\!-\!\!- 
& \ovalbox{1--} & \barrr{1.2in}
%{\!\!\! \line(1,0){.3in}} \!\!\!} 
& \ovalbox{1-1}
\end{array}
$$
\junk{
$$\begin{array}{ccccc}
-21 & \longleftarrow\quad -2- \quad\longrightarrow & 12- & & \\
\uparrow & & \uparrow & & \\
--1 & \longrightarrow & 1-- & \longleftarrow & 1-1
\end{array}
$$
(The orientation on the edges indicates the order relation $X^S$.)
}}

\centerline{\epsfig{file=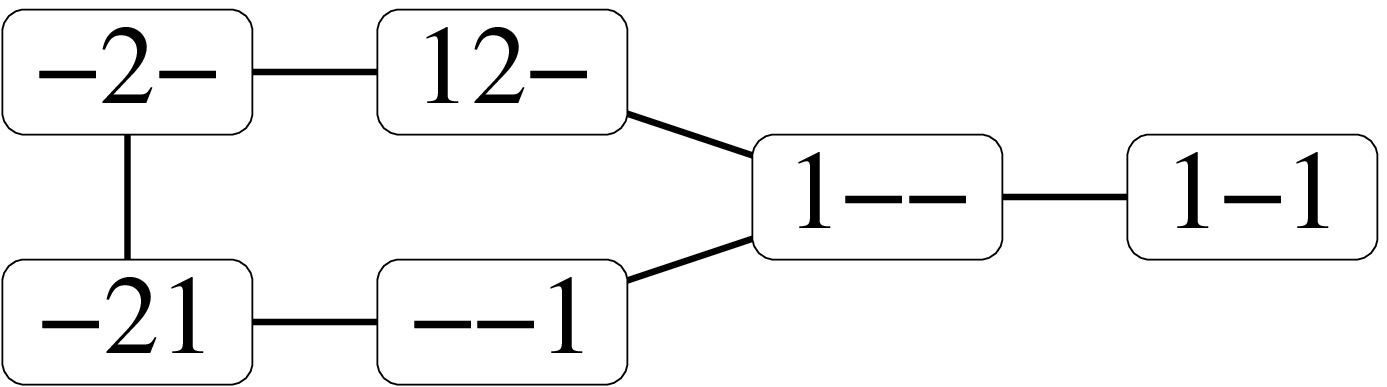,height=.8in}}

This $\Delta(X,S)$ is {\em not} homeomorphic to a ball, though
it is Cohen-Macaulay. We do not know how often this latter conclusion holds.

If the opposite B-B stratification is used ($z\to\infty$ rather than 
$z\to 0$, switching the roles of $bot$ and $top$), 
it turns out that there is only the one cone point $\,---$. 
This $\Delta(X,S^{-1})$, depicted below without its cone point, 
is again Cohen-Macaulay though not homeomorphic to a ball.

\centerline{\epsfig{file=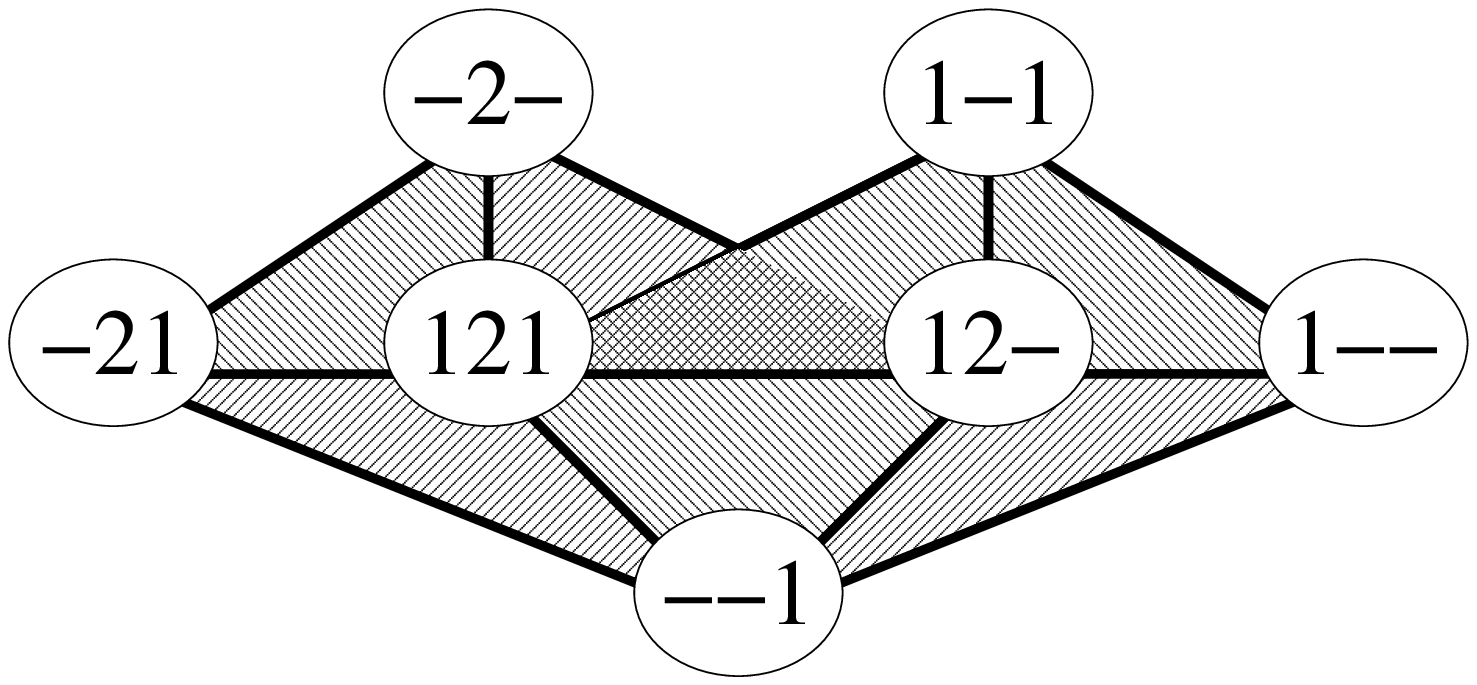,height=1.5in}}
}

\subsubsection{The punctual Hilbert scheme of four points in $\complexes^2$}
The Hilbert scheme of $n$ points in the complex plane is defined very 
concretely as the set of ideals in $\complexes[x,y]$ of codimension $n$.
It is, miraculously, smooth (Fogarty's theorem) and has received a lot
of attention recently, such as in our reference \cite{Ha}.

The subscheme in which the $n$ points all sit at the
origin is even more concrete: each ideal contains $(x,y)^n$, so can
be considered an ideal in %the finite-dimensional ring 
$\complexes[x,y]/(x,y)^n$, and hence a point in the Grassmannian
$\Gr_{n\choose 2} (\complexes[x,y]/(x,y)^n)$. This subscheme turns out
to be irreducible (another miracle), though not smooth.

The $T$ acting is the diagonal matrices from $GL_2(\complexes)$, which
acts on the ring and hence on the set of ideals. The fixed points are
the ideals $I$ generated by monomials, and are indexed by partitions of
$n$ as follows: 
the set of pairs $\{(a,b) : x^a y^b \notin I\} \subset \naturals^2$
is automatically a partition.

We draw the $T$-moment polytope for this action on the punctual Hilbert
scheme of $4$ points at the origin of the plane, labeling each vertex
by its partition. We put edges to indicate the $T$-invariant $\Pone$s,
though we won't make direct use of them.

\centerline{\epsfig{file=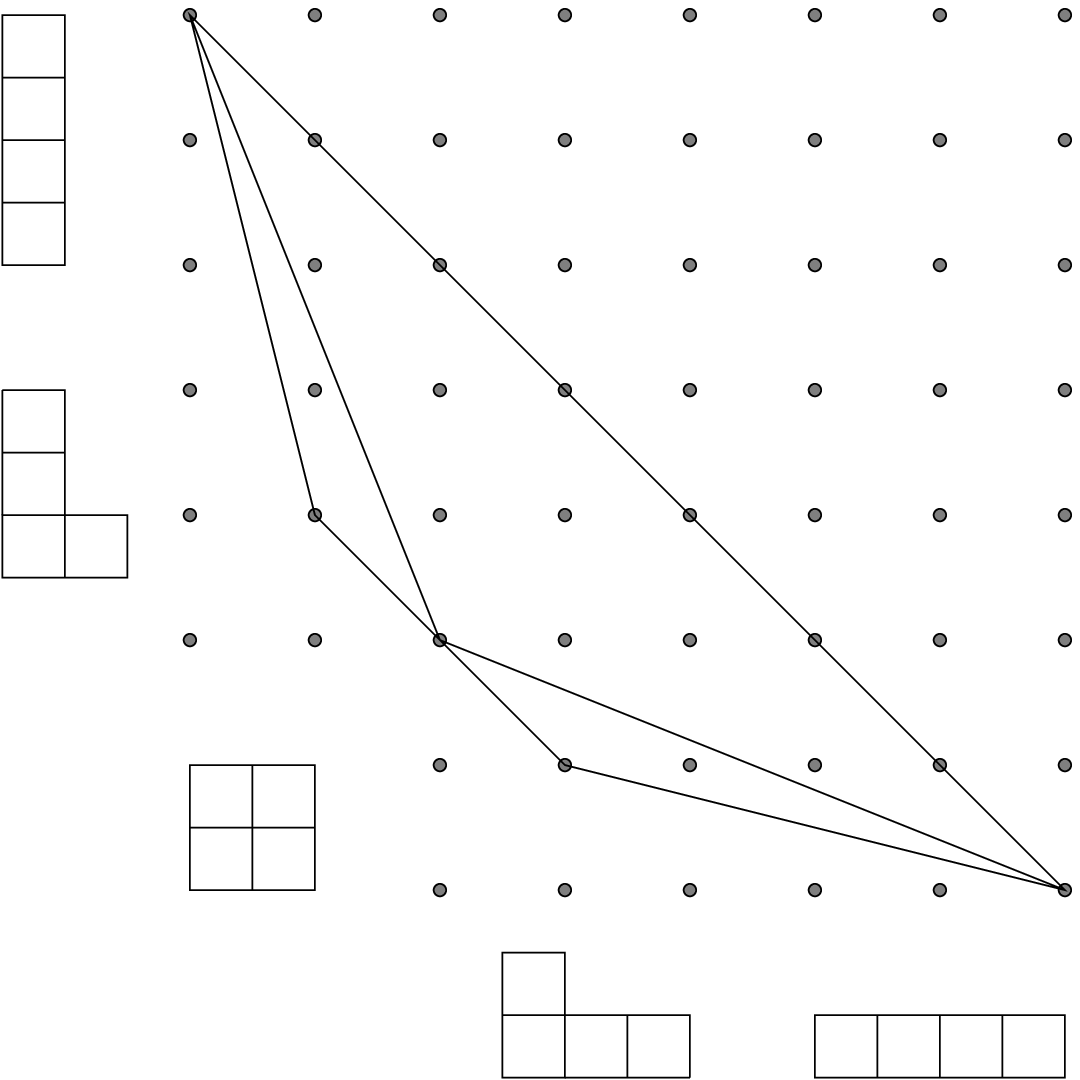,height=1.5in}}

From Northwest to Southeast, these vertices are the ideals $(y^4,x)$,
$(y^3,xy,x^2)$, $(y^2,x^2)$, $(y^2,xy,x^3)$, and $(y,x^4)$.

Let $S: \complexes^\times \to GL(2)$ be the one-parameter subgroup
$z \mapsto \diag(z,z^2)$, so $(y^4,x)$ is the top and $(y,x^4)$ the bottom.
We now describe the closures of the 
B-B strata on the Hilbert scheme:\footnote{The
obvious terminology for these is ``Gr\"obner basins''.} 
$$ \overline{X_{(y^4,x)}} = \{ (y^4,x) \} \qquad\qquad
   \overline{X_{(y^3,xy,x^2)}} = \{ (y^4, xy, x^2, Ax + By^3) \} $$
$$ \overline{X_{(x^2,y^2)}} = \{ (y^4, xy^2, x^2, Ax + By^2 + Cxy + Dy^3,
 Axy + By^3) \} $$
$$ \overline{X_{(y^2,xy,x^3)}} = \{ (y^3,y^2 x,y x^2,x^3, Bx^2 + Cxy + Dy^2,
  Ex^2 + Fxy + Gy^2)\} $$
$$ \overline{X_{(y,x^4)}} = X $$
where not both of $A,B$ are zero, and $(B,C,D),(E,F,G)$ are linearly
independent.
As these strata are not smooth, these claims are harder to check,
but we do not take space to do so here.
%since we intend to address the case of general $n$ in a future publication
%we do not take time with them here.
%

Since $X$ is irreducible, $bot = {(y,x^4)}$ is a cone point in
the complex $\Delta(X,S)$, 
but $top$ is not one as $top \notin \overline{X_{(y^2,xy,x^3)}}$.
The complex is pictured below, without the cone point:

\centerline{\epsfig{file=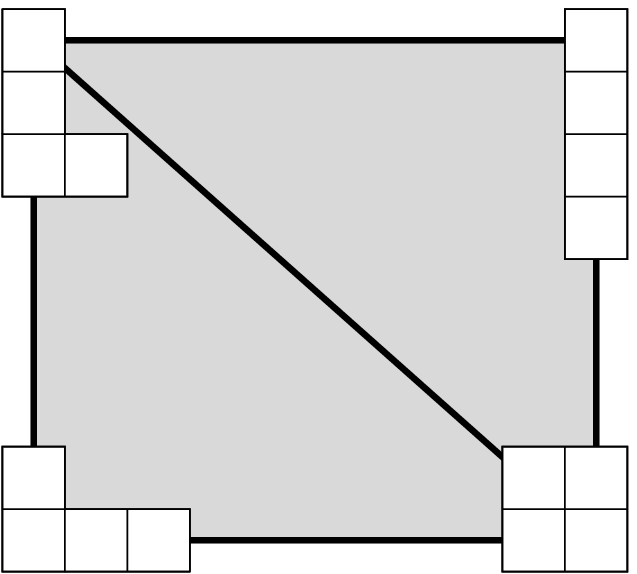,height=1in}}

\subsection{The coefficients $v_\gamma$}

We now give a recurrence on a family $\{ v(Z)_{(f_0,\ldots,f_k),Y} \}$
of natural numbers, where $Y$ is an irreducible component of
$\overline{Z_{f_0,\ldots,f_k}}$ of codimension $k$ in $Z$, in order to
give a quick definition of the $\{v_\gamma\}$.

If $X$ is irreducible, it has a unique open B-B stratum, and we denote
the fixed point in that stratum $\min(X)$. 
Lemma \ref{lem:hesselink} then implies the following: 
there exists a unique $T$-invariant hyperplane section of $X$ supported 
on $X \setminus X_{\min(X)}$. It is at this point that the projective
embedding of $X$ is finally felt: the {\em scheme} structure on this
hyperplane section gives multiplicities on the components 
of $X \setminus X_{\min(X)}$, and these multiplicities 
are building blocks in the definition of the $\{v_\gamma\}$.

\begin{Theorem}\label{thm:vrecurrence}
  Let $Z$ vary over the class of $T$-invariant subschemes of $\PP V$
  with $Z^T$ isolated, and $Y$ over the irreducible components
  of $\overline{Z_{f_0,\ldots,f_k}}$
  of codimension $k$ in $Z$. (There will only exist such $Y$ if
  $(f_0,\ldots,f_k)$ is a closure chain, and usually not even then.)
  
  Then there exists a unique assignment 
  $\left\{
    (Z,(f_0,\ldots,f_k),Y) \mapsto v(Z)_{(f_0,\ldots,f_k),Y} \in \naturals
    \right\}$
  satisfying the following conditions:
  \begin{enumerate} 
  \item For $0 \leq j \leq k$ (though we will only use $j=0,1$):
    $$v(Z)_{(f_0,\ldots,f_j,\ldots,f_k),Y} 
    = \sum_{Y_j \subseteq \overline{Z_{f_0,\ldots,f_j}},\ Y_j\supseteq Y}
    v(Z)_{(f_0,\ldots,f_j),Y_j} \ v(Y_j)_{(f_j,\ldots,f_k),Y}$$
    where the sum is over components $Y_j$ of 
    $\overline{Z_{f_0,\ldots,f_j}}$ of codimension $j$ in $Z$.
  \item $v(Z)_{(\min(Z)),Y}$ is the multiplicity of $Y$ as a component of $Z$.
  \item If $Z$ is reduced and irreducible, and $Y$ is a component
    of $Z \setminus Z_{\min(Z)}$, then $v(Z)_{(\min(Z),\min(Y)),Y}$
    is the multiplicity of $Y$ in the $T$-invariant hyperplane section of $Z$
    supported on $Z \setminus Z_{\min(Z)}$.
  \end{enumerate}
  The coefficients $v_\gamma$ from theorem \ref{thm:myDH} can be calculated 
  as $v_{\gamma=(f_0,\ldots,f_{\dim X})} = v(X)_{\gamma,\{f_{\dim X}\}}$.
\end{Theorem}

Unfortunately these multiplicities $v(Z)_{(\min(Z),\min(Y)),Y}$ can be
very difficult to compute in examples, particularly if $Z$ is singular
at $\min(Y)$. In section \ref{ssec:linrels} we prove some
linear relations on the $\{v_\gamma\}$ to help constrain them.

\subsection{Integrating more general classes}\label{ssec:myDHgen}

One of the advances of \cite{AB,BV} was to give a formula for integrating
more general classes than just $\exp(\widetilde\omega)$.
When $X$ is a symplectic manifold, 
one application of this equivariant integration is to perform 
{\em ordinary} integration on symplectic/GIT quotients $X //_p T$ of $X$.
Recall \cite{Kirwan} that for $p$ a regular value of the moment map $\Phi_T$, 
there is a surjective {\em Kirwan map} $\kappa : H^*_T(X) \onto H^*(X//_p T)$,
whose kernel can be computed by computing integrals on $X//_p T$.
That can be done as follows:

\begin{Proposition}\label{prop:D-Hgeneral}\cite{Guillemin}
  Assume the setup $X,\Phi_T$ of the Duistermaat-Heckman theorem,
  and let $p \in \Phi_T(X)$ be a regular value, 
  so the symplectic reduction $X //_p T$ is an orbifold with its own 
  symplectic form $\omega_p$. Let $\alpha \in H^*_T(X)$. 

  Then the Fourier transform of $\int_X \alpha \exp(\tomega)$
  is a measure supported on $\Phi_T(X)$, equal to Lebesgue measure 
  times a polynomial in a neighborhood of $p$,
  whose value at $p$ is $\int_{X //_p T} \kappa(\alpha) \exp(\omega_p)$.
  In particular, if $\deg(\alpha) = \dim (X //_p T)$
  then the Fourier transform is a piecewise constant function
  (times Lebesgue measure)
  whose value at $p$ is $\int_{X //_p T} \kappa(\alpha)$.
\end{Proposition}

This is also used as \cite[theorem 3.2]{GM}.

The case $\alpha \in H^*_T(pt)$, studied in \cite{GScoeffs},
is already interesting, even though
$\int \alpha \exp(\widetilde\omega) = \alpha \int \exp(\widetilde\omega)$. 
In this case, to compute the Fourier transform, we can first compute
the D-H measure and then apply the differential operator 
$\widehat{\alpha\cdot}$ that is Fourier dual to multiplication by $\alpha$.
Since $DH(X,T)$ is a piecewise-polynomial times Lebesgue measure,
this distribution can be very complicated along the breaks;
it will thus be very convenient for us that the proposition above only
requires that we understand it at {\em generic $p$}.

We now give a formula for these distributions at generic $p$, in the same
terms as in theorem \ref{thm:myDH}, and afterward discuss the case
of general $\alpha \in H^*_T(X)$, or really $\alpha \in A^T_*(X)$.
% By the above fact it is enough 
% to handle the case $\alpha \in A^*_T(pt)$, pioneered in \cite{GScoeffs}.
%Honesty compels us to explain up front
%how much {\em more} computationally intricate are the corresponding
%coefficients in this following theorem.

\newcommand\ulv{{\underline v}}
Given a list $\ulv := (v_1,\ldots,v_n)$ of vectors in $V$
and a number $k\in \naturals$, 
define a {\dfn partial fractions schema} as an injection
$\sigma: \{1,\ldots,k\} \into \{1,\ldots,n\}$
such that
$(v_j : j \neq \sigma(1),\ldots,\sigma(k))$ spans $V$
(so in particular $k\leq n-\dim V$),
and that for each $i=1,\ldots,k$, one has that $\sigma(i)$ is in 
the lex-first basis in
$(v_j : j \neq \sigma(1),\ldots,\sigma(i-1))$.

Given also a subset $M\subseteq \ulv$ with $n-k$ elements,
define the {\dfn partial fractions $k$-tensor} $\tau_{\ulv,M}$ as the sum 
$$ \tau_{\ulv,M} 
= \sum_{\sigma \atop \sigma(\{1,\ldots,k\}) = \ulv \setminus M}
v^{\sigma(1)} \tensor \cdots \tensor v^{\sigma(k)} 
\quad \in (V^*)^{\tensor k} $$
over all partial fractions schemata $\sigma$ whose image is the
complement of $M$, where $v^{\sigma(i)} \in V^*$ denotes the dual basis
element to $v_{\sigma(i)}$ in that lex-first basis 
in $\{v_j : j \notin \sigma(1,\ldots,i-1)\}$.

Finally, given a pair $(\gamma' \subseteq \gamma)$ of closure chains,
with $\gamma' = (f_0,\ldots,f_k)$ and $|\gamma| = 1+\dim X$,
let $\tau_{\gamma',\gamma}$ 
denote the partial fractions $k$-tensor $\tau_{(v_i),M}$
where $V = T^* \oplus \integers D$, $v_i = D + \Phi_T(f_i)$, 
and $M = \gamma'$.

\begin{Theorem}\label{thm:myDHgen}
  Assume the setup of theorem \ref{thm:myDH},
  and associate the same positive integers $v_\gamma$ 
  to closure chains of length $1 + \dim X$.
  To each closure chain $\gamma'$ with $1+\dim X-k$ elements,
  we associate the tensor
  $$ v_{\gamma'} := \sum_{\gamma \supseteq \gamma'}
  v_\gamma\, \tau_{\gamma',\gamma},
  \qquad\qquad
  \gamma \in \Delta(X,S),\ |\gamma| = 1+\dim X $$
  where $\tau_{\gamma',\gamma}$ is the partial fractions tensor
  defined above.

  Let $(\alpha_1,\ldots,\alpha_k) \in T^* \iso A^1_T(pt)$
  be a list of $T$-weights, so $\prod_{i=1}^k \alpha_i \in A^k_T(pt)$
  is a homogeneous class of degree $k$.
  Then near any point of $p\in \lie{t}^*$ in general position,
  the Fourier transform of multiply-by-$\alpha$,
  applied to $DH(X,T)$,
%  $\int_X \alpha \exp(\tomega) = \alpha \int_X \exp(\tomega)$ can
  can be calculated as
  $$ \sum_{\gamma' \in \Delta(X,S)\atop |\gamma'|=1+\dim X - k} 
  \left( v_{\gamma'} 
    \cdot (\alpha_1 \tensor \cdots \tensor \alpha_k) \right) 
  (C_{\gamma'})_*
  \left( \text{Lebesgue measure on the standard $(n-k)$-simplex} \right) $$
  where $C_\gamma$ is the unique affine-linear map $\reals^{n-k} \to \lie{t}^*$
  taking the vertices of the standard simplex 
  to $\{ \Phi_T(f) : f \in \gamma'\}$. (``General position'' means here
  that $p$ does not lie in the convex hull of fewer than $1+\dim T$
  elements of $\Phi_T(X^T)$.)
  
  In particular, to determine the value at a point $p$ in general
  position, we need only sum over those $\gamma'$ such that $p$ lies
  in the convex hull of $\{ \Phi_T(f) : f \in \gamma'\}$.
\end{Theorem}

\newcommand\hatX{\widehat X}

The next step beyond $\alpha\in H^*_T(pt)$ is $\alpha$ 
of the form $\sum_i \alpha_i [X_i]$, 
where $\alpha_i \in H^*_T(pt)$ and each $X_i \subseteq X$ is 
a $T$-invariant oriented submanifold. Then
$$ \int_X \alpha \exp(\tomega) 
= \int_X \left( \sum_i \alpha_i [X_i] \right) \exp(\tomega)
= \sum_i \alpha_i \int_X  [X_i] \exp(\tomega)
= \sum_i \alpha_i \int_{X_i} \exp(\tomega). 
$$
The Chow setting that we work in for the rest of the paper is closer to
equivariant homology than cohomology, and has a very appealing feature
\cite[theorem 2.1]{eqvtChow}:
{\em every} class $\alpha \in A_*^T(X)$ is of the form 
$\sum_i \alpha_i \cap [X_i]$,
for $\{X_i\}$ some $T$-invariant subvarieties.
Since theorem \ref{thm:myDHgen} makes no smoothness assumption,
it can be applied to the $X_i$ individually. 

We admit here that the statement of theorem \ref{thm:myDHgen} is
probably too unwieldy to see much direct use. We included it mainly to 
emphasize that, thanks to \cite[theorem 2.1]{eqvtChow}, an
analogue of theorem \ref{thm:myDH} for general classes $\alpha \in A_*^T(X)$
follows in some sense automatically from the $\alpha=1$ case
already treated.

\section{Background on D-H measures and B-B decompositions}

In this section we assemble some results, well-known to the experts, 
on Duistermaat-Heckman measures and Bia\l ynicki-Birula decompositions,
making little claim to originality. The closest reference we could
find for the B-B results was \cite{Hesselink}.

\subsection{D-H measures and equivariant Chow theory}

We first recast the calculation of the D-H measure of $X$ in terms 
of the equivariant Chow class of the affine cone $\hatX$. % over $X$.
This is desirable largely in that it lets us trade $X$'s multiple
fixed points for a single fixed point at the origin
(though even when $X$ is smooth, $\hatX$ won't be,
so we can't use arguments that depend on smoothness).
The base ring $A_T^*(pt) \iso \Sym(T^*) \iso H^*_T(pt)$ is the same,
and the intuition and results for Chow classes are well developed.
%Over $\complexes$, one could equally well work
%with equivariant Borel-Moore homology.
Our references for equivariant Chow theory are \cite{eqvtChow,eqvtHandA}.

{\em All} of our equivariant Chow classes will live on vector spaces.
When $Z\subseteq W$ for $W$ a vector space, we will write $[Z\subseteq W]$
for the corresponding class in $A^*_T(W) \iso \Sym(T^*)$. 
Usually $W$ will be our
ambient space $V$, and then we will denote $[Z\subseteq V]$ simply by $[Z]$.

The only facts we need about equivariant Chow classes are these
trivial generalizations from ordinary Chow theory:

\begin{Proposition}\label{prop:Chowfacts}
  Let a torus $U$ act on a vector space $V$, preserving a subscheme $Y$
  and a hyperplane $H = \{b=0\}$, 
  with $\lambda \in U^*$ the $U$-weight on the line $V/H$.
  For any $U$-invariant subscheme $Z\subseteq V$, 
  let $[Z] \in A^*_U(V)$ denote the associated equivariant Chow class.
  \begin{itemize}
  \item If $Y$ contains no components in $H$, i.e. if $b$ is not a
    zero divisor on $Y$, \\
    then $[H \cap Y] = [H][Y] = \lambda\, [Y]$. \\
    Conversely, 
    if $Y\subseteq H$ then $[Y] = [H] [Y \subseteq H]$,
    and $= [H] [Y\times L]$ where 
    and $L$ is a $U$-invariant complement in $V$ to $H$.
  \item If $\{Y_i\}$ are the top-dimensional components of $Y$,
    occurring with multiplicities $\{m_i \in \naturals\}$, then
    $[Y] = \sum_i m_i [Y_i]$.
  \item If there exists a closed subscheme $F \subseteq V\times S$
    whose projection to the connected base $S$ is flat, 
    and whose fibers are $U$-invariant subschemes
    two of whom are $Y$ and $Y'$, then $[Y] = [Y']$.
  \end{itemize}
\end{Proposition}

The condition in the first is very easy to check when $Y$ is reduced
and irreducible, and the second lets one reduce to that case.
Joseph described this same recursion on his polynomials in \cite{Joseph}.

In fact we will work with the $(T\times \Gm)$-action on $V$,
where the multiplicative group $\Gm$ acts by rescaling, 
with weight denoted $D$.
Our base ring is thus the larger
polynomial ring $A^*_{T\times \Gm}(pt) = \Sym(T^*)[D]$,
and all the weights $\{D+\lambda, \lambda \in T^*\}$ live
in an open half-space, making it easy to define Fourier transforms.

\begin{Proposition}\label{prop:Aclass}
  Let $X \subseteq \PP V$ be a projective scheme invariant under
  the action of a torus $T$ on $V$. Let $\hatX \subseteq V$ be the
  affine cone over $X$, considered as a $(T\times \Gm)$-space.

  Let $e_{\vec 0} \hatX := [\hatX] \big/ [\vec 0]$ be 
  the {\dfn equivariant multiplicity} \cite{Rossmann} of $\hatX$ (at $\vec 0$),
  where $[\hatX],[\vec 0] \in A^*_{T\times \Gm}(V)$
  denote the $(T\times \Gm)$-equivariant classes, 
  and $e_{\vec 0} {\hatX}$ lives in the ring of fractions.
  (The denominator $[\vec 0]$ is 
  the product of the weights of $T\times \Gm$ on $V$.)

  Assume now for convenience that $T$ acts locally freely on $X$.
  (We can achieve this by breaking $X$ into components, and 
  quotienting $T$ by the kernel of the action.)
  Let $f_X : \lie{t}^* \times \reals \to \reals$ be the 
  piecewise-homogeneous-polynomial function
  such that $f_X(\vec v,r) = 0$ for $r\leq 0$,
  $f_X(\vec v,r) = r^{\dim X - \dim T} f_X(\vec v/r,1)$ for $r>0$, 
  and $f_X(\vec v,1)$ is the Duistermaat-Heckman function.
  Then $f_X$ and $e_{\vec 0} \hatX$ are related by Fourier transform.
\end{Proposition}

\begin{proof}
  This is an easy version of \cite[theorem 2.1]{Rossmann}, though 
  that is stated for the more difficult complex-analytic case.
  
  One cheap proof in our algebro-geometric setting here is to note
  that both $[\hatX]$ and $DH(X,T)$ are constant in locally free
  $T$-equivariant families, such as provided by Gr\"obner degenerations
  to monomial schemes, and both behave additively under the decomposition
  of $X$ into its top-dimensional components $X_i$ with multiplicities $m_i$.
  The components of monomial schemes are $T$-invariant
  linear subspaces.
  We are thus reduced to checking the easy case that $\hatX \leq V$ is a
  $T$-invariant linear subspace; both sides become $1 / \prod (D+\lambda)$ 
  where $\lambda$ runs (with multiplicity) over the the $T$-weights in 
  the vector space $\hatX$.
\end{proof}

In section \ref{ssec:linrels}, we will use the following proposition
to constrain the coefficients $\{v_\gamma\}$.

\begin{Proposition}\label{prop:congruence}
  Continue the notation of proposition \ref{prop:Aclass}.

  Let $f\in X^T$, and assume that the $\Phi_T(f)$-weight space $L$ in $V$
  is one-dimensional (i.e. $\PP L = f$). 
  Let $C_f X \subseteq T_f \PP V$ denote the normal cone at $f$ in $Y$
  and the tangent space at $f$ to $\PP V$, respectively.
  The ($T$-invariant) tangent cone carries a Chow class 
  $[C_{f} X \subseteq T_f \PP V] \in A^*_T(T_{f} \PP V) \iso \Sym(T^*)$.
  Denote by $[\{\vec 0\} \subseteq T_{f} X] \in A^*_T(T_{f} X)$
%  $[T_{f} X \leq T_{f} \PP V] \in A^*_T(T_{f} \PP V)$ 
  the evident Chow class (a product of $T$-weights).

  Then specializing the following rational functions in $A^*_{T\times \Gm}(pt)$
  at $D = -\Phi_T(f)$, we have
  $$ [\hatX] / [0\times L] \equiv [C_f X \subseteq T_{f} \PP V] 
  / [\{\vec 0\} \subseteq T_{f} \PP V] $$
  where neither side involves division by $0$.
\end{Proposition}

\newcommand\Hom{{\rm Hom}}

\begin{proof}
  Let $T' \leq T\times \Gm$ be the pointwise stabilizer of $L$, so
  $(T')^*$ can be naturally
  identified with $(T^* \times \integers D) / \langle D + \Phi_T(f)\rangle$.
  Let $H$ be the unique $T$-invariant complement to $L$, so
  $(0,1) \in H\times \{1\} \subseteq V$ 
  provides a $T'$-invariant model for an open neighborhood of $f\in \PP V$.

  We can regard the flat degeneration of $\hatX$ to $C_{L} \hatX$
  (whose relation to $C_f X$ we discuss in a moment) as an
  embedded degeneration inside $V$, as follows. 
  Let $Q : \Gm \to GL(H\oplus L)$ 
  act by $Q(z)\cdot (h,\ell) := (zh,\ell)$. Then the flat limit
  $\lim_{z\to\infty} Q(z)\cdot \hatX$ is easily identified with 
  $C_{L} \hatX$.
  By proposition \ref{prop:Chowfacts}, we get an equation
  $$ [\hatX] = [C_L \hatX] \quad \in A^*_{T\times \Gm}(V). $$
  We can similarly identify the flat limit
  $\lim_{z\to\infty} Q(z)\cdot ( \hatX \cap (H\times \{1\}) )$ with
  $C_{(0,1)} ( \hatX \cap (H\times \{1\}) )$.
  Using the $(Q\times T')$-equivariant model $H\times \{1\}$ above,
  this can in turn be $T'$-equivariantly identified with $C_f X$.

  Consider now the projection $\pi : C_L \hatX \onto L$, where each
  fiber $\pi^{-1}(\ell)$ is a subscheme of $H$.
  Since $\pi$ is $\Gm$-equivariant, the fibers are constant except for
  possibly the fiber over $0$. So $C_L \hatX$ is supported on 
  $\left(\pi^{-1}(1) \times L\right) \cup \left(\pi^{-1}(0) \times 0\right)$.
%  and we have identified $\pi^{-1}(1)$ with $C_f X$,
%  but only $T'$-equivariantly. 
  Therefore
  $$ [C_L \hatX] = [\pi^{-1}(1) \times L] + (D+\Phi_T(f))q \quad $$
  for some polynomial $q \in A^*_{T\times \Gm}(H)$, where the factor
  $D+\Phi_T(f)$ comes from $[0 \in L]$. (This $q$ is not necessarily
  the class of $\pi^{-1}(0)$, but a sum over its components, with some
  multiplicities we will not determine.) 
  Since $H\oplus L\onto L$ is $Q$-invariant, 
  $\pi^{-1}(1) \iso C_{(0,1)} ( \hatX \cap (H\times \{1\}) )
  \iso C_f X$.
  Chaining these together, and working modulo $D+\Phi_T(f)$, we get
  $$ [\hatX] \equiv [C_f X \subseteq T_{f} X] \quad\bmod \ D+\Phi_T(f).$$
  We can $T'$-equivariantly
  identify $T_f \PP V \iso \Hom(L,V/L) \iso V/L$.
  Dividing both sides of this last equation by the $T'$-weights
  in that space produces the formula we seek.
\end{proof}

\subsection{B-B decompositions}\label{ssec:oldBB}

In the next few lemmas
we will study B-B decompositions using $S$-orbit closures.

%The following technical lemma is only used in the later lemma
%\ref{lem:hesselink}.

\newcommand\barX{\overline X}
\begin{Lemma}\label{lem:bbvanishing}
  Let $b\in V^*$, thought of as an element of $\Gamma(X; \mathcal O(1))$,
  be an $S$-weight vector of weight $k \in \integers$. 
  Let $f\in X^S$ be a fixed point,
  and recall $\Phi_S(f) \in \integers$ denotes the weight 
  of $S$ on $\calO(1)|_f$.
  \begin{itemize}
  \item If $k > \Phi_S(f)$, then $b$ vanishes at $f$.
  \item If $k < \Phi_S(f)$, then $b$ vanishes on {\em all} of
    $\overline X_f$.
  \item If $k = \Phi_S(f)$, then on $\barX_f$, $b$ is unique up to scale.
  \item If $b$ does not vanish at $f$ (so $k=\Phi_S(f)$), 
    then $b$ does not vanish on $X_f$. 
  \end{itemize}
\end{Lemma}

\begin{proof}
  We start with the case $X = \Pone$, $f = \vec 0$, 
  and therefore $X_f = \Pone\setminus\vec \infty$.
  Let $h$ be the order of the global stabilizer subgroup scheme
  $\{z\in \kk^\times : S(z)\cdot \vec 1=\vec 1\}$.
  An $S$-equivariant line bundle $\mathcal L$ on $\Pone$ is classified by its
  degree $d$, and the $S$-weight on the fiber over $0$, 
  in this case $\Phi_S(f)$.
  Then the weights in the representation 
  $\Gamma(\Pone; \mathcal L)$ are
  $(\Phi_S(f),\Phi_S(f)+h,\ldots,\Phi_S(f)+dh)$,
  and each weight space is $1$-dimensional. 

  In particular, if $k<\Phi_S(f)$, the weight $k$ does not occur in
  this space of sections. So $b$ is the zero section.
  This proves the second statement (still for $X=\Pone$).

  For the others, note that the weight $\Phi_S(f)+ih$ section vanishes
  at $f$ to order $i$. This proves the first statement, and this plus
  the $1$-dimensionality together prove the third. 
  For the fourth, note that the only $S$-covariant section that doesn't
  vanish at $f$ is the $i=0$ one, which vanishes only at $\vec\infty$,
  hence not on $X_f$. This settles $X=\Pone$.
  
  Now we consider the case of general $X$. Let $x$ be a point of $X_f$.
  Define an $S$-equivariant map $\kk^\times \to {\mathbb P}V$ by
  $z \mapsto S(z)\cdot x$, and use the projectivity of $X$ 
  to extend to an $S$-equivariant map 
  $\Pone \to {\mathbb P}V$ (which takes $0\mapsto f$ since $x\in X_f$).
  Pull back $\calO(1)$ to $\Pone$ and apply the previous analysis. 
\end{proof}

Most of the published results about B-B decompositions concern the
case that $X$ is smooth, or at least normal, 
with the following as a rare exception:

\begin{Lemma}\cite[section 2]{Konarski}\label{lem:konarski}
  Define the {\dfn opposite B-B decomposition $X = \coprod_f X^f$} 
  using the inverse action of $S$ on $X$, $S'(z) := S(z^{-1})$.
  Then for each $f\in X^S$, $\dim X^f + \dim X_f \geq \dim X$.
\end{Lemma}

Konarski also handles the case when $X^S$ is not isolated,
which gives an extra term we may omit. He only states the lemma 
(as a corollary to theorem 3, the normal case) for the case $X$ 
irreducible (or at least, ``a variety''),
but this generalizes easily: 
when $X = \Union_i X(i)$ is the decomposition into irreducible components
then $X_f = \Union X(i)_f, X^f = \Union X(i)^f$, 
so 
\begin{eqnarray*}
   \dim X^f + \dim X_f 
   &=& \max_i \dim X(i)^f + \max_j \dim X(j)_f \\
&\geq& \max_i \left(\dim X(i)^f + \dim X(i)_f\right) 
\geq \max_i \dim X(i) \qquad 
\text{the irreducible case} \\
&=& \dim X. 
\end{eqnarray*}

To see a (normal) example where the inequality is strict, tilt a
square pyramid $P$ up on one edge, and let $f$ be the apex, with
$\overline{P_f},\overline{P^f}$ both being triangles.  
Then the inequality is $2 + 2 > 3$.

Say that $X$ has a {\dfn unique supporting fixed point} 
if $X = \overline{X_f}$ for some $f\in X^T$. This will be part of a more 
general definition in the next section, but is an important enough
special case that we introduce the notation $\min(X) = f$ for it.
If $X \neq \overline{X_f}$ for any $f\in X^T$, then $\min(X)$ is undefined.

Most authors using B-B decompositions remark somewhere 
that if $X$ is irreducible, it has a unique
supporting fixed point, called the {\dfn sink}.  
(Proof: exactly one B-B stratum $X_f$ is
open, and $X$ is the closure of that $X_f$.)
Irreducibility is an unnatural condition for us,
as any nonempty $\overline{X_{f_0,\ldots,f_i}}$ also
has a unique supporting fixed point, $f_i$, though it may be reducible
even when $X$ itself is irreducible (see the tilted octahedron example in
section \ref{ssec:TVs}).

\begin{Corollary}\label{cor:downwardPone}
  Let $X$ have a unique supporting fixed point, 
  and let $f \in X^S, f \neq \min(X)$.

  Then there exists a map $\beta : \Pone \to X$, $S$-equivariant with
  respect to the standard action of $\Gm$ on $\Pone$, 
  such that $\beta(\infty) = f \neq \beta(0)$. 
  Moreover
  $$ \Phi_S(\beta(\infty)) - \Phi_S(\beta(0)) 
  = \deg \beta \cdot \deg \beta(\Pone) $$
  meaning the degree of the map $\beta$ to its image,
  times the projective degree of its image curve.

  In particular each $f\neq \min(X)$ has $\Phi_S(f) > \Phi_S(\min(X))$. 
\end{Corollary}

\begin{proof}
  The assumption on $X$ says $X = \overline{X_{\min(X)}}$, and the
  assumption on $f$ 
  says $X_f \subseteq \overline{X_{\min(X)}} \setminus X_{\min(X)}$. 
  Hence $\dim X_f < \dim X$. 
  By lemma \ref{lem:konarski}, $\dim X^f > 0$, 
  so there exists a point $x\in X^f \setminus \{f\}$,
  automatically {\em not} $S$-invariant.
  Define the map $\beta : \Pone \to \overline{X^f}$ by extending
  $$ \beta: z \mapsto S(z)\cdot x, \quad z\in \Gm $$
  so $\beta(\infty) = f$ by choice of $x$.
  Then by the same analysis as in lemma \ref{lem:bbvanishing} 
  (and with the same notation $h,d$),
  $\Phi_S(\beta(\infty)) - \Phi_S(\beta(0)) = hd$. 

  This shows that for each $f\neq \min(X)$,
  there exists some other $g \in X^S$ (namely $g=\beta(0)$) such that
  $\Phi_S(f) > \Phi_S(g)$.  By induction on the finite set $\Phi_S(X^T)$, 
  for each $f\neq \min(X)$ we have $\Phi_S(f) > \Phi_S(\min(X))$.
\end{proof}

The assumption of projectivity is very clearly necessary here, since otherwise
$X$ could be $\Pone$ with $0$ and $\infty$ identified. (The reason that
$X$ is sometimes assumed to be normal, as in much of \cite{Konarski}, 
is to ensure that $T$-invariant affine open sets on it possess closed
equivariant affine embeddings.)

\begin{Lemma}\label{lem:span}
  Assume $X$ has a unique supporting fixed point $\min(X)$. 
  Let $W\leq V$ be the smallest linear subspace containing $\hatX$.
  Then the $\Phi_T(\min(X))$-weight space in $W$ is $1$-dimensional.
\end{Lemma}

\begin{proof}
  Obviously we may shrink $V$ to $W$ from the outset.

  To see that the weight space is nonzero, consider
  $h \in V^*$ as an element of $\Gamma(X; \calO(1))$,
  and choose an $h$ that does {\em not} vanish at $\min(X)$. 
  (We know such an $h$ exists because $X$ is projectively embedded, 
  rather than merely carrying an ample line bundle.)
  Expand $h$ as a sum of $T$-weight vectors; at least one of them must
  not vanish at $\min(X)$, and let $b$ be that term. 
  Note that we can determine the $T$-weight of $b$ -- 
  it must be the $T$-weight on the line $\calO(1)|_f$.
  
  By assumption $X = \overline{X_{\min(X)}}$. 
  Then the last conclusion of lemma \ref{lem:bbvanishing} gives us the
  uniqueness of $b$ up to scale.
\end{proof}

An even smaller $W$ will be used in proposition \ref{prop:branch}.

Results like the following are often attributed to \cite{Hesselink}
(at least for $X$ smooth irreducible), but I was not able to locate an
exact reference therein. The last part is quite close to \cite[theorem 3]{BB}
(again, only stated for the smooth case, though his proof generalizes).

\begin{Lemma}%\cite{Hesselink}
\label{lem:hesselink}
  Let $X \subseteq \PP V,T,S$ be as in theorem \ref{thm:myDH}.
  Assume $X$ has a unique supporting fixed point.
  Then there is a $T$-invariant hyperplane $\PP H$ in $\PP V$
  not containing $\min(X)$, and the subscheme $\PP H \cap X$ does
  not depend on the choice of $\PP H$. As a set,
  $\PP H \cap X = \Union_{f\neq \min(X)} X_f$.
\end{Lemma}

\begin{proof}
%  Instead of thinking about $T$-invariant hyperplanes $\PP H$ in $\PP V$, 
%  we think about $T$-weight vectors $b$ in $V^*$, considered up to scale,
%  where $\PP H$ is $b$'s vanishing set.
%  Then $\PP H \not\ni \min(X)$ exactly if $b$ doesn't vanish at $\min(X)$.
%
  Existence of the desired $\PP H$, 
  or equivalently, of a $T$-weight vector $b$ not vanishing at $\min(X)$,
  is given by lemma \ref{lem:span}, which also gives the uniqueness
  of $\PP H \cap X$.
  
%  Now assume $f \in X^T, f \neq \min(X)$.  
  By corollary \ref{cor:downwardPone}, $\Phi_S(f) > \Phi_S(\min(X))$
  for each $f\neq \min(X)$.
  Then by lemma \ref{lem:bbvanishing}, $b$ vanishes on $\overline{X_f}$,
  and doesn't vanish on $X_{\min(X)}$.
  Hence $\PP H \cap X = \Union_{f\neq \min(X)} X_f$ as a set.
\end{proof}

\junk{
  The assumption in lemma \ref{lem:hesselink} that $X$ is irreducible
  seems unnecessary, **and is gone**
  but is good enough for our single use in 
  lemma \ref{lem:positive}, in which we show that the coefficients
  $v_\gamma$ are strictly positive. That lemma is not necessary for
  the validity of the formula in theorem \ref{thm:myDH}; 
  it is merely to help justify our definition of $\Delta(X,S)$.
}

It is really in this lemma that the assumption of isolated fixed
points becomes crucial. Thanks to this lemma,
to cut down from $X$ to the union of smaller B-B strata
(or a scheme supported thereon) it suffices to take a hyperplane section,
which by proposition \ref{prop:Aclass} will let us inductively compute
equivariant Chow classes.

The following result is, in some sense, a tightest possible version of lemma
\ref{lem:span}. Essentially the same idea was used 
in \cite[proofs of theorem 17 and corollary 19]{eqvtHandA}.
We won't need it for the proofs of the main theorems,
but it will appear in section \ref{ssec:linrels}.
%and to greater effect in \cite{AKfuture}.

\begin{Proposition}\label{prop:branch}
  Let $W\leq V$ be the linear span of the points $X^T$,
  and pick a $T$-equivariant projection $\hat\beta: V\onto W$.
  Then the induced map $\beta: X \to \PP W$ is well-defined 
  (has no basepoints), finite, and $T$-equivariant.
\end{Proposition}

\begin{proof}
  If $\vec v \in \hatX \setminus \vec 0$, 
  then $\PP\vec v \in X_f$ for some $f\in X^T$.
  Let $K := \ker \hat\beta$.
  Since the line over $f$ is not contained in $K$,
  there is an element of $K^\perp \leq V^*$ not vanishing at $f$,
  and hence (as explained in the proof of lemma \ref{lem:span})
  a $T$-weight vector $b_f \in K^\perp$ not vanishing at $f$.
  By lemma \ref{lem:bbvanishing} the function $b_f$
  doesn't vanish on $X_f$.

  Hence the subscheme $\{\vec v \in \hatX : \langle b, \vec v\rangle = 0
  \ \forall b\in K^\perp\}$
  is supported at the origin (and therefore of finite length), which shows
  the lack of basepoints. Since the map $\hatX \to W$ is 
  dilation-equivariant, the fiber over $\vec 0$ is the largest fiber,
  which shows the finiteness of the map.
  The $T$-equivariance is clear. 
\end{proof}

When $X$ is reduced, the map $\beta: X \to \PP W$ is termed a
{\em branchvariety} of $\PP W$ in \cite{branch}, where 
we studied families
of such maps. This will also be the point of view in \cite{AKfuture}.

Theorem \ref{thm:ChowDH} below will be a formula for the
equivariant multiplicity $[\hatX]/[\vec 0]$, with a surprisingly
small actual denominator.
As in \cite[proofs of theorem 17 and corollary 19]{eqvtHandA},
one can use proposition \ref{prop:branch} to predict already 
that the denominator divides $\prod_{f\in X^T} (D+\Phi_T(f))$,
though to carry this out would involve introducing some definitions
(e.g. the D-H measures of modules and cycles) we do not take space for here.

\section{Proofs of the main theorems}
\label{sec:proofs}

Throughout section \ref{ssec:myBB} we work with algebraic sets, rather
than schemes, and do not bother to include the caveat ``as a set''
after each claimed equality.

\subsection{Supporting fixed points and closure chains
  in the B-B decomposition}\label{ssec:myBB}

Call $f\in X^S$ a {\dfn supporting fixed point} 
if the B-B stratum $X_f$ contains an open set in $X$.
Since the B-B decomposition is into finitely many strata, one of them
must contain an open set, so every B-B decomposition has a supporting
fixed point.
In the Stanley-Reisner case $X = X(\Delta)$, the point $i$ is a supporting
fixed point iff there exists a facet (meaning, a maximal face) $F \in \Delta$
with $\min(F) = i$.

When $X$ has a unique supporting fixed point (e.g. $X$ irreducible), 
$X_{\min X}$ is actually open in $X$, 
rather than merely containing an open set.
But more generally this can fail: for an example
let $X = \Proj \complexes[x_1,x_2,x_3] \big/ \langle x_1 x_3 \rangle$ 
be the Stanley-Reisner scheme of a union of two intervals, and $f=2$. 
(Perhaps the term ``sink'' should be reserved for those $f$ with
open $X_f$.)
If we had required this more restrictive condition in the definition
of supporting fixed point, we wouldn't have lemmas \ref{lem:uniquesupport} 
or \ref{lem:irrsupport}.

\begin{Lemma}\label{lem:uniquesupport}
  Let $F \subseteq X^T$ be the set of supporting fixed points.
  Then $X = \overline{\Union_{f\in F} X_f} = \Union_{f\in F} \overline{X_f}$. 

  In particular, if $X$ has only one supporting fixed point $\min(X)$, 
  then $X = \overline{X_{\min(X)}}$ (matching the terminology from
  section \ref{ssec:oldBB}). 
\end{Lemma}

\begin{proof}
  The proof is pure point-set topology.
  If $\overline{X\setminus Y} \subsetneq X$, then 
  $Y$ contains
  the nonempty open set $X \setminus \overline{X\setminus Y}$.
  Contrapositively, if $Y_1 = Y$ doesn't contain an open set in $X$,
  then $X\setminus Y_1$ is dense in $X$.

  If $Y_2 \subseteq X$, $Y_2 \cap Y_1 = \nulset$ also contains no open set
  in $X$, then it contains no open set in $X \setminus Y_1$, 
  so $X \setminus (Y_1\cup Y_2)$ is dense in $X\setminus Y_1$,
  hence dense in $X$.

  Repeating this, we can remove finitely many subsets that each contain
  no open set in $X$, with the remainder still dense in $X$. 
  Hence $X \setminus \Union_{f \notin F} X_f$ is dense in $X$.
  By the B-B decomposition, this subset is $\Union_{f \in F} X_f$.

  Finally, $X \supseteq \Union_{f\in F} \overline{X_f} \supseteq
  \overline{\union_{f\in F} X_f} = X$, hence all three are equal.
\end{proof}

\begin{Lemma}\label{lem:inherit}
  Let $Y \subseteq X$ be closed and $S$-invariant, e.g. if $Y$
  is an irreducible component of $X$. 
  Then $Y$ has a B-B decomposition, with $Y_f = Y \cap X_f$ for $f\in Y^S$
  (as noted in \cite{BB}).
  Each closure chain $\gamma$ for $Y$ is a closure chain for $X$, 
  with $\gamma \subseteq Y^S$.
\end{Lemma}

\begin{proof}
  Since $\Gm$ is connected, its action on the set of components of $X$
  is trivial, which is why irreducible components are $S$-invariant.
  The next claim is tautological: 
  $$ Y_f 
  = \left\{y\in Y\ :\ \lim_{z\to 0} S(z)\cdot y = f \right\}
  = \left\{y\in X\ :\ y\in Y,\ \lim_{z\to 0} S(z)\cdot y = f \right\} 
  = Y \cap X_f. $$
  Obviously the closure chains $\gamma$ for $Y$ have $\gamma \subseteq Y^S$,
  and $\overline{Y_\gamma} \subseteq \overline{X_\gamma}$; 
  thus each closure chain for $Y$ is a closure chain for $X$.
\end{proof}

The converse is not true:
it is often the case that $\gamma \subseteq Y^S$ is not a closure chain for 
$Y\subseteq X$ even though it is a closure chain for $X$, and this can happen 
even when $Y$ is irreducible (consider $Y=F_1$ as in section \ref{ssec:TVs}, 
with $X = Y \cup \Pone$ intersecting at the points $a,c$). Our best
partial converse will be corollary \ref{cor:deltaunion} below.

\begin{Lemma}\label{lem:irrsupport}
  If $Y\subseteq X$ is an irreducible component, then its unique
  supporting fixed point $\min(Y)$ is also a supporting fixed point of $X$.
  In particular, every irreducible component of $\overline{X_f}$
  contains $f$.
\end{Lemma}

\begin{proof}
  If $Y$ is a component, it contains an open set $Y^\circ$ in $X$, so
  it must meet some $X_f$ for $f$ a supporting fixed point, 
  and we may pick $y \in Y^\circ \cap X_f$.
  Since $Y$ is closed and $S$-invariant,
  $\lim_{z\to 0} S(z)\cdot y = f$ lies in $Y$.

  Since $Y^\circ \cap X_f$ contains a nonempty open set in $Y$ (irreducible), 
  it is dense. It is contained in $Y \cap X_f$, 
  which by lemma \ref{lem:inherit} is $Y_f$, and this makes $f$ a
  supporting fixed point of $Y$. Since $Y$ is irreducible,
  it is the unique such.
\end{proof}

The second half of the following lemma is very similar to one in \cite{BB},
where it is only proven
under the assumption that each intersection $X_f \cap X^g$ is transverse.

\begin{Lemma}\label{lem:dimdecrease}
  Assume $X$ has a unique supporting fixed point, % $\min(X) = bot$,
  and let $f \in X^S$, $f\neq \min(X)$. Then $\dim \overline{X_f} < \dim X$.
  Consequently, any closure chain for $X$ has at most $1 + \dim X$ elements.
\end{Lemma}

\begin{proof}
  This can be proven using lemma \ref{lem:hesselink}, or even more directly,
  as we do now.
  $$ X_f \subseteq X \setminus X_{\min(X)} 
  = \overline{X_{\min(X)}} \setminus X_{\min(X)} $$
  and the right-hand side has lower dimension than $X_{\min(X)}$.
  Then $\dim \overline{X_f} = \dim X_f < \dim X_{\min(X)}$.

  Now let $(f_0,\ldots,f_m)$ be a closure chain for $X$. 
  Then 
  $$ \overline{X_{f_0}} \supsetneq \overline{X_{f_0,f_1}} \supsetneq
  \ldots \supsetneq \overline{X_{f_0,\ldots,f_m}} $$
  where by its construction, $\overline{X_{f_0,\ldots,f_i}}$ has a
  unique supporting fixed point $f_i$. Hence by the above, 
  the dimensions of these spaces are strictly decreasing in this chain,
  and there must therefore be at most $1 + \dim X$ of them.
\end{proof}

\newcommand\bargamma{{\overline{\gamma}}}

To define the coefficients $v_\gamma$ of theorem \ref{thm:myDH}, 
we will need a refinement of the notion of closure chain,
which we develop in a series of lemmas.

\begin{Lemma}\label{lem:witnesses}
  Let
  $\bargamma = \left((f_0 \in Y_0),(f_1 \in Y_1),\ldots,(f_m \in Y_m)\right)$
  be a list of pairs $(f_i\in X^T, Y_i \subseteq X)$ 
  such that each $Y_{i\geq 0}$ is an irreducible component 
  of $\overline{(Y_{i-1})_{f_i}}$, interpreting $Y_{-1}$ as $X$.
  Assume also that $\gamma = (f_0,\ldots,f_m)$ is nonrepeating.

  Then $\gamma = (f_0,\ldots,f_m)$ is a closure chain, and
  we call $\bargamma$ a {\dfn witness to $\gamma$ in $X$}.

  If $Y \subseteq X$ is a component, then any witness to a
  closure chain in $Y$ is a witness to the same closure chain in $X$.
  Conversely, any witness 
  $\bargamma = \left((f_0 \in Y_0),\ldots,(f_m \in Y_m)\right)$ 
  to a closure chain in $X$ is a witness in $Y$ too, if $Y\supseteq Y_0$.
\end{Lemma}

\begin{proof}
  We need to show that $\overline{X_{f_0,\ldots,f_m}} \neq \nulset$.
  So we show inductively that 
  each $Y_k \subseteq \overline{X_{f_0,\ldots,f_k}}$.
  First, 
  $$ Y_0 = \overline{(Y_0)_{f_0}} = \overline{(Y_0) \cap X_{f_0}} 
  \subseteq \overline{X_{f_0}} \qquad \text{using lemma \ref{lem:inherit}.} $$
  Then for $i>0$, using lemma \ref{lem:inherit} and induction,
  $$ Y_i \subseteq \overline{(Y_{i-1})_{f_i}} 
  = \overline{Y_{i-1} \cap X_{f_i}} 
  \subseteq \overline{\overline{X_{f_0,\ldots,f_{i-1}}} \cap X_{f_i}}
  = \overline{X_{f_0,\ldots,f_i}}. $$
  So $\overline{X_{f_0,\ldots,f_m}} \supseteq Y_m$ and hence is
  nonempty, making $\gamma$ a closure chain.

  Note that the sequence $(Y_i)$ is weakly decreasing, since
  $ Y_i \subseteq \overline{(Y_{i-1})_{f_i}} 
  \subseteq \overline{Y_{i-1}} = Y_{i-1}$. 
  Hence the conditions $Y_i \subseteq X$ are equivalent to $Y_0 \subseteq X$.
  The final condition is that $Y_0$ is a component of $\overline{X_{f_0}}$.

  The second and third claims are then tautological, as the definition of
  witness only involves the ambient space in the condition $Y_0 \subseteq X$.
\end{proof}

We make two remarks about the definition. The $(f_i)$ in a witness can
be recovered from the $(Y_i)$ as $f_i = \min(Y_i)$, but it seems
unnatural to leave the $(f_i)$ out of the definition as it doesn't
simplify the axioms on the $(Y_i)$.  Also, one could formulate a weaker
notion of witness, a chain of varieties in which each $Y_i$ is
$S$-invariant and irreducible with $\min(Y_i) = f_i$, just not
necessarily a component.  But the
following lemma suggests that we will not need this greater generality.

\begin{Lemma}\label{lem:existwitnesses}
  Every closure chain $\gamma = (f_0,\ldots,f_m)$ has witnesses, 
  and only finitely many thereof.
%  $\bargamma = \left((f_0 \in Y_0),(f_1 \in Y_1),\ldots,(f_m \in Y_m)\right)$.
\end{Lemma}

\begin{proof}
  The proof of existence is by induction on $m$. If $m=0$ then
  this is easy: pick some component $Y_0$ of $\overline{X_{f_0}}$,
  and apply lemma \ref{lem:irrsupport} to know that $f_0 = \min(Y_0)$.
  
  Now assume $m>0$.  Tautologically, $(f_1,\ldots,f_m)$ is a closure
  chain of $\overline{X_{f_0,f_1}}$, and so has a witness
  $\left((f_1 \in Y_1),\ldots,(f_m \in Y_m)\right)$ by induction.
  Since $Y_1$ is irreducible, we may choose an irreducible component $Y_0$ 
  of $\overline{X_{f_0}}$ containing it. By lemma \ref{lem:irrsupport}, 
  $f_0 = \min(Y_0)$.

  It remains to show that $Y_1 \subseteq \overline{(Y_0)_{f_1}}$.
  (It will automatically be a component, since it is a component
  of the larger $\overline{X_{f_0,f_1}}$.)
  Since $f_1 = \min(Y_1)$, it is
  enough to show $(Y_1)_{f_1} \subseteq (Y_0)_{f_1} = Y_0 \cap X_{f_1}$
  (by lemma \ref{lem:inherit}),
  and indeed we know that $Y_1 \subseteq Y_0$
  and $(Y_1)_{f_1} \subseteq X_{f_1}$. 
  
  If one imagines enumerating witnesses to a given closure chain by
  picking $Y_0$, then $Y_1$, etc., then at each stage one picks an
  irreducible component of a projective scheme, which means finitely
  many choices. (Sometimes the scheme is empty and there are zero
  choices, if one has made a bad choice along the way; this is why
  we didn't use this argument to show existence.)
\end{proof}

\begin{Corollary}\label{cor:deltaunion}
  If $X = \Union_i X_i$ is the decomposition into irreducible components,
  then $\Delta(X,S) = \Union_i \Delta(X_i,S)$.
\end{Corollary}

\begin{proof}
  For each closure chain $\gamma \in \Delta(X,S)$, pick a witness
  $\bargamma$, and an irreducible component $X_i$ of $X$ containing 
  the $Y_0$ from $\bargamma$. 
\end{proof}

Since $DH(X,T) = \sum_{X_i} DH(X_i,T)$ (for $\{X_i\}$ the
top-dimensional primary components), and theorem \ref{thm:myDH} gives each 
side of this equation as a sum over top-dimensional faces of the corresponding 
complexes, one might expect this union of the complexes to be
disjoint on the top-dimensional faces.
It need not be, as the second example in section \ref{sssec:tricky} shows.
In \cite{AKfuture}, it will indeed be a disjoint union of
some simplicial {\em posets} that refine the simplicial complexes
presented here.

Our main interest in witnesses is in the case $m=\dim X$, as these are the 
only $\gamma$ that contribute in the formula in theorem \ref{thm:myDH}.

\begin{Lemma}\label{lem:maxwitnesses}
  Let $\bargamma = \left((f_0 \in Y_0),\ldots, %(f_1 \in Y_1),\ldots,
    (f_{\dim X} \in Y_{\dim X})\right)$
  be a witness. Then $Y_0$ is a top-dimensional component of $X$,
  each $Y_{i+1}$ is a Weil divisor in $Y_i$, and
  $Y_{\dim X}$ is the singleton $\{f_{\dim X}\}$.
\end{Lemma}

\begin{proof}
  By lemma \ref{lem:dimdecrease} $\dim Y_i + 1 \leq \dim Y_{i-1}$,
  so $\dim Y_i + i\leq \dim Y_0$. 
  Hence 
  $$ \dim X 
  \leq \dim Y_{\dim X} + \dim X 
  \leq \dim Y_i + i
  \leq \dim Y_0 + 0 \leq \dim X,$$
  making each one an equality: $\dim Y_i = \dim X - i$.

  In particular, $Y_{\dim X}$ is $0$-dimensional. It is also irreducible,
  and contains $f_{\dim X}$. %, so it has no other points.
\junk{
  This settles the first two claims.
  Since $\overline{(Y_i)_{f_{i+1}}} \subsetneq Y_i$ and $Y_i$ is irreducible, 
  $\overline{(Y_i)_{f_{i+1}}}$
  has lower dimension than $Y_i$. Since $\overline{(Y_i)_{f_{i+1}}}$
  contains $Y_{i+1}$ by definition of witness,
  its dimension must be exactly $\dim Y_i - 1$, so $Y_{i+1}$ has
  the right dimension to be a component.
}
\end{proof}

As mentioned earlier, for purposes of computing the D-H measure of $X$
we may assume $X$ is equidimensional. Under that assumption, we now
show (though we won't make use of it) that {\em maximal} closure
chains are {\em maximum}, i.e. have $1+\dim X$ elements.

\begin{Proposition}\label{prop:equidim}
  If $X$ is equidimensional, then so is $\Delta(X,S)$ (the first as a
  reduced scheme, the second as a simplicial complex).

  In particular, if the B-B decomposition is a stratification,
  then the poset $(X^T, \geq)$ is a {\em ranked} poset, with the ranking
  given by $f \mapsto \dim X - \dim X_f$.
\end{Proposition}

\begin{proof}
  Let $\gamma$ be a maximal closure chain, and pick a witness
  $\bargamma = \left((f_0 \in Y_0), \ldots, (f_m \in Y_m)\right)$ to
  $\gamma$ using lemma \ref{lem:existwitnesses}. We wish to show $m=\dim X$.

  We claim $Y_0$ must be a component of $X$.
  For otherwise, we could pick a component $Y\subseteq X$ properly
  containing it, and stick $\min(Y)$ at the beginning of $\gamma$,
%  to extend the witness and its closure chain, 
  contradicting $\gamma$'s maximality.
  By $X$'s equidimensionality, $\dim Y_0 = \dim X$.

  As in the proof of lemma \ref{lem:dimdecrease},
  $\dim Y_i \leq \dim X - i$ for each $i$.
  We claim now that this is an equality. 
  Otherwise, let $i$ be the least such that the inequality is strict;
  by the previous paragraph we know $i>0$.
  So $\dim Y_i < \dim Y_{i-1} - 1$.
  By lemma \ref{lem:hesselink}, $Y_i$ is contained inside the
  hyperplane section
  $Y_{i-1} \setminus (Y_{i-1})_{f_{i-1}}$.
  So we can pick a component $Z$ of $Y_{i-1} \setminus (Y_{i-1})_{f_{i-1}}$
  containing $Y_i$. 
  Since $Y_{i-1} \setminus (Y_{i-1})_{f_{i-1}}$ 
  is pure of codimension $1$ inside $Y_{i-1}$, $Z$ properly contains $Y_{i-1}$.
  Now $\left( (f_0\in Y_0),\ldots, (f_{i-1} \in Y_{i-1}),
    (\min(Z) \in Z), (f_i \in Y_i), \ldots, (f_m\in Y_m) \right)$ 
  is a witness, so by interposing $\min(Z)$ we have extended $\gamma$, 
  contradiction.

  Finally, we claim $Y_m$ is a point. Otherwise, the hyperplane section
  $Y_m \setminus (Y_m)_{f_m}$ is nonempty, so we can pick a component
  $Y_{m+1}$ of it and extend $\gamma$ at the end, contradiction.

  Hence $m = \dim X - \dim Y_m = \dim X - 0$, as was to be shown.
\end{proof}

This proposition is another sign of the minimality of $\Delta(X,S)$,
in the following sense.
Theorem \ref{thm:myDH} only makes use of the top-dimensional faces
of $\Delta(X,S)$. Proposition \ref{prop:equidim} says (in the case
that one has thrown out
the lower-dimensional components) that $\Delta(X,S)$
only has those faces implied by those top-dimensional ones, with no %other 
extraneous maximal-but-not-maximum faces.

In \cite{AKfuture} we will give a degeneration-based proof of
proposition \ref{prop:equidim}, which will enable us to prove the following
additional result: if $X$ is equidimensional and connected in
codimension one (e.g. if $X$ is irreducible), then so too is $\Delta(X,S)$.

We made special mention in proposition \ref{prop:equidim} of the
stratification case, as one can use this to show the known but perhaps
surprising fact that the poset of $K$-orbits on a flag manifold $G/B$,
for $K$ a symmetric subgroup of $G$, is a ranked poset.
This poset is also that of the $B$-orbits on $G/K$,
which is an order ideal in the poset of $B$-orbits on the wonderful
compactification of $G/K$ \cite{wonderful}.  Those orbits are given by
a B-B decomposition, and this proposition then provides the proof.

\subsection{The main theorems}\label{ssec:proofs}

We first define a refinement of the $\{v_\gamma\}$, using the
witnesses $\bargamma = \left((f_0 \in Y_0),(f_1 \in Y_1),\ldots,
  (f_{\dim X} \in Y_{\dim X})\right)$ to $\gamma$.
Hereafter in this section, $\bargamma$ will denote a witness
$\left((f_0 \in Y_0),(f_1 \in Y_1),\ldots, (f_k \in Y_k)\right)$ with
the condition $\dim Y_i = \dim X - i$ for $i=0,\ldots,k$, though not until
later will we assume $k=\dim X$.

Let $m_{\bargamma,0}$ denote the multiplicity of $Y_0$ as a component
of $X$. % (which it is, by lemma \ref{lem:maxwitnesses}).
For each $i>0$, use lemma \ref{lem:hesselink} to choose some
$T$-invariant $\PP H$ that misses $f_{i-1}$, with which to (well-)
define the subscheme $Y_{i-1} \cap \PP H$.
This hyperplane section is equidimensional 
of dimension $\dim Y_{i-1} - 1 = \dim Y_i$,
%(again using lemma \ref{lem:maxwitnesses}),
and we can let $m_{\bargamma,i}$ denote
the multiplicity of $Y_i$ as a component of it. (In a moment we will prove
that it is indeed a component, i.e. that $m_{\bargamma,i} > 0$.)
Then define
$$ v(X)_\bargamma := \prod_{i=0}^k m_{\bargamma,i}
\qquad \text{and} \qquad
 v_{\gamma = (f_0,\ldots,f_{\dim X})} 
:= \sum_{\bargamma = \{(f_0 \in Y_0),\ldots,(f_{\dim X} \in Y_{\dim X})\}}
v_\bargamma. $$
(The latter sum is a finite sum by lemma \ref{lem:existwitnesses}.)
\junk{ In proposition \ref{prop:assembling} we will give
  a recurrence among these $v_\bargamma$. }

\junk{
  It is tempting to refer to $m_{\bargamma,i}$ as ``the order of vanishing
  of $b$ along $Y_i$'', where $b$ is the equation of the hyperplane $\PP H$.
  However, if $Y_{i-1}$ is not normal along $Y_i$, this definition is slightly 
  problematic, as the second example in section \ref{sssec:tricky} shows; 
  really that $m_{\bargamma,i} = 3$ is the sum of two orders $1+2$ of vanishing
  along the line and the conic. 
}

\begin{Lemma}\label{lem:positive}
  The numbers $\{m_{\bargamma,i}\}$, $\{v(X)_\bargamma\}$, 
  and $\{v_\gamma\}$ are all strictly positive.
\end{Lemma}

\begin{proof}
  Fix a witness $\bargamma$ and an $i>0$.
  Pick an $S$-invariant hyperplane $\PP H$ not containing $f_{i-1}$.
  By lemma \ref{lem:hesselink},
  $Y_{i-1} \cap \PP H \supseteq \overline{(Y_{i-1})_{f_i}}$,
  and that contains $Y_i$. As explained above, the dimensions match
  so $Y_i$ is a component of $Y_{i-1} \cap \PP H$.
  This shows that $m_{\bargamma,i} > 0$.
  
  (In fact this is the principal place that we use algebraic
  geometry/Chow theory rather than topology/homology, where the
  singularities made the orientation issues look particularly fearsome.)
  
  Thus each $v(X)_\bargamma$ is a product of positive integers, hence
  positive.  By lemma \ref{lem:existwitnesses}, each $\gamma$ has some
  witness $\bargamma$, thus $v_\gamma$ is a nonempty sum of positive integers,
  hence positive.
\end{proof}

That $m_{\bargamma,0}$ has such a different definition from
$m_{\bargamma,i>0}$ is a hint that $X$ should perhaps be required 
to be reduced from the beginning, as in \cite{branch}; it will indeed be so
in \cite{AKfuture}.

We first prove an analogue of theorem \ref{thm:myDH} for equivariant
Chow classes, and then give the straightforward equivalence with the
stated theorem. Theorem \ref{thm:myDHgen} will also be a reasonably
automatic consequence.

\begin{Theorem}\label{thm:ChowDH}
  Let $X \subseteq \PP V, T, S$ be as in theorem \ref{thm:myDH}.
  Let $\hatX \subseteq V$ be the $(T\times\Gm)$-invariant affine cone
  over $X$, and $[\hatX] \in A^*_{T\times\Gm}(V) \iso \Sym(T^*)[D]$ its 
  equivariant Chow class. 
  Then its equivariant multiplicity can be computed as follows:
  $$ [\hatX]/[\vec 0]= \sum_\gamma \frac{v_\gamma}
  {\prod_{f \in \gamma} \left(D + \Phi_T(f)\right)} $$
  considered as an element of the fraction field of the polynomial ring
  $\Sym(T^*)[D]$,
  where $\gamma \in \Delta(X,S)$ varies over the maximum-length closure
  chains, and the $v_\gamma$ are as defined above.
\end{Theorem}

\begin{proof}
  By the definition of the $v_\gamma$, the formula is obviously
  equivalent to the more refined sum over maximum-length witnesses
  $$ [\hatX]/[\vec 0]=
  \sum_{\bargamma = \{(f_0 \in Y_0),\ldots,(f_{\dim X} \in Y_{\dim X})\}}
                                % \sum_\bargamma 
  \frac{v_\bargamma}  {\prod_{i=0}^n \left(D + \Phi_T(f_i)\right)}. $$
  %  where the $(f_i)$ are the fixed points listed in $\bargamma$.

  The interesting case is when $X$ is reduced and irreducible; as we now show, 
  it is easy to handle the general case if granted this special one.

  Let $\{\hatX_i\}$ be the top-dimensional irreducible components of $\hatX$
  (similarly $X_i$ of $X$),
  occurring with multiplicities $m_i$.
  For each witness $\bargamma$ in $X$, 
  by lemmas \ref{lem:witnesses} and \ref{lem:maxwitnesses}
  $\bargamma$ is a witness in $X_i$ iff $Y_0 = X_i$.
  Let $v^i_\bargamma \in \naturals$ denote the coefficient in 
  the (assumed) formula for $[\hatX_i]$ if $\hatX_i = Y_0$,
  and $0$ otherwise.
  Unwinding the definitions, we see $m_i v^i_\bargamma = v_\bargamma$
  for $\hatX_i = Y_0$, and is $0$ otherwise;
  thus $\sum_i m_i v^i_\bargamma = v_\bargamma$.
  \junk{
    By lemma \ref{lem:inherit}, the complex of closure chains
    of each $X_i$ includes into that of $X$. 
    So for any closure chain $\gamma$ of $X$,
    let $v^i_\gamma = 0$ if $\gamma$ is {\em not} a closure chain of $X_i$,
    and let $v_\gamma = \sum_i m_i v^i_\gamma$. 
    By corollary \ref{cor:deltaunion}, 
    $\gamma$ is a closure chain for some $X_i$, so $v_\gamma > 0$.}
  Then 
  \begin{eqnarray*}
    [\hatX]/[\vec 0]
%    [\hatX] 
    &=& \sum_i m_i\, [\hatX_i]/[\vec 0]
    \qquad\text{by proposition \ref{prop:Chowfacts}}\\
    &=& \sum_i m_i \sum_\bargamma \frac{v^i_\bargamma}
    {\prod_{i=0}^n \left(D + \Phi_T(f_i)\right)}
    = \sum_\bargamma \frac{\sum_i m_i v^i_\bargamma}
    {\prod_{i=0}^n \left(D + \Phi_T(f_i)\right)}
    = \sum_\bargamma \frac{v_\bargamma}
    {\prod_{i=0}^n \left(D + \Phi_T(f_i)\right)}
  \end{eqnarray*}
  as claimed. In each sum $\bargamma$ varies over witnesses in $X$ having the
  (by lemma \ref{lem:dimdecrease}) maximum length, $1+\dim X$.
  
  Now assume that $X$ is reduced and irreducible, and that the theorem
  has been proven in dimensions $< \dim X$ (for both irreducible and
  reducible). Since $X$ is irreducible, it has a unique supporting
  fixed point. % we will call $bot$.
  
  By lemma \ref{lem:hesselink},
  there exists a $T$-invariant hyperplane $\PP H \leq \PP V$
  not containing $\min(X)$, whose defining equation $b=0$ 
  is of $(T\times \Gm)$-weight $D+\Phi_T(\min(X))$.
  Since $\PP H \not\ni \min(X)$ and $X$ is irreducible, $\PP H$ contains no
  component of $X$.
  So by proposition \ref{prop:Chowfacts}
  $$ \big(D + \Phi_T(\min(X))\big)\, [\hatX] = [H \cap \hatX]. $$
  This $H \cap \hatX$ has dimension $\dim X - 1$ 
  (in fact it is equidimensional), 
  so by induction its equivariant multiplicity has a formula of the form
  $$ [H \cap \hatX]/[\vec 0]
  = \sum_\gamma \frac{v'_\gamma}
  {\prod_{i=0}^n \left(D + \Phi_T(f_i)\right)} $$
  where $\gamma$ varies over 
  the maximum-length closure chains of $\PP H\cap X$. (In this formula
  we write $v'$ rather than $v$ because the formula is for $H\cap \hatX$,
  not $\hatX$.)
  
  Since $X$ is irreducible, by lemma \ref{lem:maxwitnesses} $Y_0 = X$
  in any maximum witness in $X$. Hence the maximum witnesses in $\PP H\cap X$
  and $X$ correspond $1$:$1$ under the map
  $$ \alpha: \left((f_1 \in Y_1),\ldots,(f_{\dim X}\in Y_{\dim X})\right) 
  \quad \mapsto \quad
  \left((\min(X)\in X),(f_1\in Y_1),\ldots,(f_{\dim X}\in Y_{\dim X})\right) $$
  Since $X$ is reduced, its multiplicity is $1$, 
  so $v'_\bargamma = v_{\alpha(\bargamma)}$. Together,
  \begin{eqnarray*}
    [\hatX]/[\vec 0]
    &=& \frac{1}{D + \Phi_T(\min(X))}\, [H \cap \hatX]/[\vec 0]  \\
    &=& \frac{1}{D + \Phi_T(\min(X))} \sum_\bargamma \frac{v'_\bargamma}
    {\prod_{i=1}^n \left(D + \Phi_T(f_i)\right)}
    = \sum_{\alpha(\bargamma)} \frac{v_{\alpha(\bargamma)}}
    {\prod_{i=0}^n \left(D + \Phi_T(f_i)\right)} 
  \end{eqnarray*}
  where the left sum is over maximum-length witnesses for $\PP H \cap X$,
  and as argued above
  the right sum is over maximum-length witnesses for $X$.
  By lemma \ref{lem:positive}, the coefficients are all positive.
%  Thus the theorem is proved for $X$.
\end{proof}

\begin{proof}[Proof of theorem \ref{thm:myDH}.]
  The rational function $\prod_{f \in \gamma} \left(D + \Phi_T(f)\right)^{-1}$
  is the specialization of $\prod_{i=0}^n x_i^{-1}$
  under the map $x_i \mapsto D + \Phi_T(f_i)$.
  Correspondingly, the Fourier transform of 
  $\prod_{f \in \gamma} \left(D + \Phi_T(f)\right)^{-1}$ is 
  the image of Lebesgue measure on $\reals^{n+1}_{\geq 0}$ under the map 
  $(\xi_0,\ldots,\xi_n) \mapsto \sum_i \xi_i \left(D + \Phi_T(f_i)\right)$,
  where $\gamma = (f_0,\ldots,f_n)$. 

  Now proposition \ref{prop:Aclass}, applied to theorem \ref{thm:ChowDH},
  gives theorem \ref{thm:myDH}.
\end{proof}

Recall that theorem \ref{thm:myDH} was stated as an existence result for
a mysterious family of coefficients $\{v_\gamma\}$, that were then
defined in theorem \ref{thm:vrecurrence}. The proof just given didn't
explicitly use theorem \ref{thm:vrecurrence}'s family of coefficients, 
but rather the $\{v_\gamma\}$ constructed by summing over witnesses
$\{v_\bargamma\}$.
To prove theorem \ref{thm:vrecurrence} we will show that these two
definitions of $\{v_\gamma\}$ agree.

\begin{proof}[Proof of theorem \ref{thm:vrecurrence}]
  We first show uniqueness, and thereby uncover a formula for
  $v(Z)_{(f_0,\ldots,f_k),Y}$ in terms of witnesses.

  The $j=0$ case of the recurrence is
  \begin{eqnarray*}
     v(Z)_{(f_0,\ldots,f_k),Y} 
    &=& \sum_{Y_0 \subseteq \overline{Z_{f_0}},\ Y_0\supseteq Y}
      v(Z)_{(f_0),Y_0} \ v(Y_0)_{(f_0,\ldots,f_k),Y}.  
  \end{eqnarray*}
  where in this and in the sums below, $Y_i$ varies over the irreducible
  components of the space said to contain it.
  Then expand the last term, using the $j=1$ case:
  \begin{eqnarray*}
    v(Z)_{(f_0,\ldots,f_k),Y} 
    &=&
    \sum_{Y_0 \subseteq \overline{Z_{f_0}},\ Y_0\supseteq Y}
    v(Z)_{(f_0),Y_0}
    \sum_{Y_1 \subseteq \overline{Z_{f_0,f_1}},\ Y_1\supseteq Y}
    v(Y_0)_{(f_0,f_1),Y_1} 
    v(Y_1)_{(f_1,\ldots,f_k),Y} \\
    &=&
    \sum_{Y_0 \subseteq \overline{Z_{f_0}},
      Y_1 \subseteq \overline{Z_{f_0,f_1}},\ Y_1\supseteq Y}
    v(Z)_{(f_0),Y_0}
    v(Y_0)_{(f_0,f_1),Y_1} 
    v(Y_1)_{(f_1,\ldots,f_k),Y}.  
  \end{eqnarray*}
  Expanding the last term using the $j=1$ expansion $k-1$ more times, we get
  \begin{eqnarray*}
    v(Z)_{(f_0,\ldots,f_k),Y} 
    &=&
    \sum_{(Y_0,\ldots,Y_k \supseteq Y) 
      : Y_i \subseteq \overline{Z_{f_0,\ldots,f_i}}}
    v(Z)_{(f_0),Y_0}
    \prod_{i=1}^k v(Y_{i-1})_{(f_{i-1},f_i),Y_i}.
  \end{eqnarray*}
  Assumptions (2) and (3) of the recurrence
  tie these to the definitions at the beginning of section \ref{ssec:proofs}:
  $$  v(Z)_{(f_0),Y_0} = m_{\bargamma,0}, \qquad
  v(Y_{i-1})_{(f_{i-1},f_i),Y_i} = m_{\bargamma,i}  $$
  and so 
  $$ v(Z)_{(f_0,\ldots,f_k),Y} 
  = \sum_{(Y_0,\ldots,Y_k \supseteq Y) 
    : Y_i \subseteq \overline{Z_{f_0,\ldots,f_i}}} 
  m_{\bargamma,0} \prod_{i=1}^k m_{\bargamma,i}
  = \sum_{(Y_0,\ldots,Y_k \supseteq Y) 
    : Y_i \subseteq \overline{Z_{f_0,\ldots,f_i}}} 
  v(Z)_{(f_0\in Y_0,\ldots,f_k\in Y_k)}. $$
  In particular, if $Z=X$ and $k=\dim X$ so (by lemma \ref{lem:maxwitnesses})
  $Y=\{f_k\}$, this says
  $ v(X)_{(f_0,\ldots,f_{\dim X}),Y} 
  = \sum_\bargamma v(X)_\bargamma =: v(X)_\gamma $, as we wanted to show.

  So far we have shown that the recurrence has at most one solution
  (even using only $j\leq 1$), 
  and that solution reproduces the $\{v_\gamma\}$ used in the proof
  of theorem \ref{thm:myDH}. It remains to show that this solution
  --- summing over all ways to lift $(f_0,\ldots,f_k)$ to a witness
  ending with $Y_k=Y$ --- actually satisfies the recurrence, but this is easy:
  extend to a witness by first choosing $Y_j$, then choose the other $\{Y_i\}$
  behind and ahead $Y_j$.
\end{proof}

\begin{proof}[Proof of theorem \ref{thm:myDHgen}.]
  Our goal is to understand 
  $$ \left( \prod_{i=1}^k \alpha_i\right)\ [\hatX] 
  = \left(\prod_{i=1}^k \alpha_i\right)\ \sum_\gamma \frac{v_\gamma}
  {\prod_{f \in \gamma} \left(D + \Phi_T(f)\right)} 
  = \sum_\gamma v_\gamma
  \frac{\prod_{i=1}^k \alpha_i}
  {\prod_{f \in \gamma} \left(D + \Phi_T(f)\right)} $$
  where the terms on the right are ready for
  multivariable partial fractions expansion.

  This will create many terms along the way of the form 
  $q / \prod_{f\in Q} \left(D + \Phi_T(f)\right)$, whose Fourier transform is
  some complicated distribution supported on the cone positively
  spanned by $\{D + \Phi_T(f) : f \in Q\}$. 
  By the assumption that $p$ is in general position, we can drop any
  such term for which that set $\{D + \Phi_T(f) : f \in Q\}$ does not
  $\rationals$-span $T^* \oplus \integers D$.

  In the first step of this expansion, we write $\alpha_1$ as a
  linear combination of the lex-first basis found in 
  $\{D + \Phi_T(f) : f \in \gamma\}$. 
  The coefficients involved are the $v^{\sigma(1)} \cdot \alpha_1$ where
  $\sigma(1)$ varies over that lex-first basis. (We are beginning to
  build partial fractions schemata $\sigma$; so far we have 
  specified the value at $1$.) That gives an initial expansion of
  $$ 
  \frac{\prod_{i=1}^k \alpha_i}
  {\prod_{f \in \gamma} \left(D + \Phi_T(f)\right)}
  = \left( {\prod_{i=2}^k \alpha_i} \right)
  \frac{\alpha_1}  {\prod_{f \in \gamma} \left(D + \Phi_T(f)\right)} 
  =   \left( {\prod_{i=2}^k \alpha_i} \right)
  \sum_{\sigma(1)} 
  \frac{v^{\sigma(1)} \cdot \alpha_1}
  {\prod_{f \in \gamma \setminus \{f_{\sigma(1)}\}} \left(D + \Phi_T(f)\right)}
  $$
  At this point we must split into cases, because the lex-first basis in
  $(D + \Phi_T(f_i)\ :\ i=0,\ldots,\dim X, i\neq \sigma(1))$ 
  depends on $\sigma(1)$.

  Each time we bring in an $\alpha_i$, we linearly expand it in the lex-first
  basis in the remaining terms in the denominator. If there
  is no such basis, then as
  explained above the term may be dropped. Partial fractions expansion then 
  eats each term from this basis in turn, 
  and the choice of which one is recorded
  as $\sigma(i)$; the coefficient incurred is $v^{\sigma(i)} \cdot \alpha_i$.
  After doing this $k$ times, the final coefficient on 
  $1 / {\prod_{f \in \gamma'} \left(D + \Phi_T(f)\right)} $
  is a sum over partial fraction schemata $\sigma$, of the product
  of $v^{\sigma(i)} \cdot \alpha_i$:
  $$ \frac{\prod_{i=1}^k \alpha_i}
  {\prod_{f \in \gamma} \left(D + \Phi_T(f)\right)} 
  = \sum_\sigma
  \frac{ \left( v^{\sigma(1)} \tensor \cdots \tensor v^{\sigma(k)} \right)\cdot
    (\alpha_1 \tensor \cdots \tensor \alpha_k) }
  {\prod_{f \in \gamma \setminus \{f_{\sigma(i)}\}} \left(D + \Phi_T(f)\right)}
  = \sum_{\gamma' \subseteq \gamma}
  \frac{ \tau_{\gamma',\gamma}\cdot (\alpha_1 \tensor \cdots \tensor \alpha_k)}
  {\prod_{f \in \gamma} \left(D + \Phi_T(f)\right)} 
  $$
  up to terms dropped because, as explained above, they don't affect
  the measure near $p$.

  We now sum over $\gamma$, 
  then Fourier transform as in the proof above of theorem \ref{thm:myDH}, 
  and we arrive at the complicated statement of theorem \ref{thm:myDHgen}.
\end{proof}

\section{Constraints on the coefficients $v_\gamma$}\label{sec:coeffs}

\subsection{An easy case of the multiplicities $v(Z)_{(f_0,f_1),Y}$}

There is an important special case in which these multiplicities
from theorem \ref{thm:vrecurrence} are easy to compute.

\begin{Proposition}\label{prop:easyPone}
  Let $\{ v(Z)_{(f_0,f_1),Y} \}$ be as in theorem \ref{thm:vrecurrence}.
  By corollary \ref{cor:downwardPone},
  there exists an $S$-equivariant map $\beta : \Pone \to X$
  such that $\beta(\infty) = f_1 \neq \beta(0)$. 

  Assume $Z$ smooth at $f_1$.
  Then the image of $\beta$ is $\overline{Z^{f_1}}$, a rational
  curve smooth away from $\beta(0)$.

  Assume further that $\beta(0) = f_0$. Then 
  $$ v(Z)_{(f_0,f_1),Y} 
  = \deg \beta(\Pone)
  = \frac{\Phi_S(f_1) - \Phi_S(f_0)}{\big| Stab_S(Z^{f_1}) \big|}
  = \frac{\Phi_T(f_1) - \Phi_T(f_0)}{-wt(T_{f_1} Z^{f_1})}$$
  where $Stab_S(Z^{f_1})$ denotes the 
  generic stabilizer subgroup scheme of $S$ acting on $Z^{f_1}$,
  and $wt(T_{f_1} Z^{f_1})$ denotes the $T$-weight on 
  the tangent line $T_{f_1} Z^{f_1}$. (The numerator is a multiple thereof.)
\end{Proposition}

\begin{proof}
  Since $Z$ is smooth at $f_1$, so are $Z_{f_1},Z^{f_1}$,
  with $\dim Z_{f_1} + \dim Z^{f_1} = \dim C$
  where $C$ is the component of $Z$ containing $f_1$, and
  the intersection $Z_{f_1} \cap Z^{f_1} = \{f_1\}$ is transverse \cite{BB}.

  Since $Z_{f_1}$ is smooth and connected, its closure is irreducible,
  and $Y$ is supposed to be an irreducible component thereof.
  Hence $Y = \overline{Z_{f_1}}$ and is smooth at $f_1$.
  Similarly $\overline{Z^{f_1}}$ is irreducible and smooth at $f_1$.
  Since $Y$ is assumed to be codimension $1$ in $Z$, 
  we infer $\dim \overline{Z^{f_1}} = 1$; 
  since $\overline{Z^{f_1}}$ contains the curve $\beta(\Pone)$
  they must be equal.

  Let $b=0$ be the equation of a $T$-invariant hyperplane missing $f_0$
  (existence guaranteed by lemma \ref{lem:hesselink}). We are attempting
  to determine the order of vanishing of $b$ along 
  $Y = \overline{Z^{f_1}}$. By the transversality
  of the intersection $\overline{Z_{f_1}} \cap \overline {Z^{f_1}}$,
  we may instead restrict $b$ to $\beta(\Pone)$, and determine the
  order of vanishing of $b$ at the point $f_1 \in \beta(\Pone)$.

  This is the degree of the curve $\beta(\Pone)$, computed in
  terms of $\Phi_S$ in corollary \ref{cor:downwardPone}. To compute
  in terms of $\Phi_T$ requires that one extend the $S$-weight analysis
  in lemma \ref{lem:bbvanishing} to the $T$-weights, which is straightforward.
\end{proof}

The ``$\beta(0)=f_0$'' condition in the proposition holds 
for flag manifolds (as follows from the next lemma), 
but is is not otherwise automatic.
If $Z = F_1$ is the example from section \ref{ssec:TVs}, and
$Y=\overline{Z_b}$, then $\overline{Z^{f_1}}$ is the $\Pone$
connecting $b$ and $c$; it doesn't make it down to $d$.

\begin{Corollary}\label{cor:oldAK}
  Let $X$ be smooth (and equidimensional),
  and assume each intersection $X_f^g := X_f \cap X^g$ is transverse.
  Then $\Delta(X,S)$ is the order complex of the poset $(X^T,\geq)$.

  Fix a maximal $\gamma$, and for each $i=1,\ldots,\dim X$, 
  assume that $\overline{X_{f_{i-1}}}$ is smooth at $f_i$.
  Then each $\overline{X_{f_{i-1}}^{f_i}}$
  is a (possibly cuspidal) rational curve, and
  $ v_\gamma = \prod_{i=1}^{\dim X} 
  \deg \overline{X_{f_{i-1}}^{f_i}}. $
\end{Corollary}

\begin{proof}
  In \cite{BB} it is proven that this transversality condition implies
  that the B-B decomposition is a stratification. 
  Hence $\overline{X_{f_0,\ldots,f_k}} = \overline{X_{f_k}}$ as long as
  $(f_0,\ldots,f_k)$ is a chain in $(X^T,\geq)$, so nonempty for each chain.
  Therefore $\Delta(X,S)$ is the order complex.

  We now show that under 
  the $S$-equivariant map $\beta: \Pone \to \overline{X_{f_{i-1}}^{f_i}}$
  constructed in proposition \ref{prop:easyPone}, we have $\beta(0)=f_i$.
  For otherwise, 
  $\overline{X_{f_{i-1}}} \supsetneq \overline{X_{\beta(0)}} 
  \supsetneq \overline{X_{f_i}}$, 
  with $\dim \overline{X_{f_i}} \leq \dim \overline{X_{f_{i-1}}} - 2$
  by lemma \ref{lem:dimdecrease}. 
  But by lemma \ref{lem:maxwitnesses}, $\dim X_{f_i} = \dim X_{f_{i-1}} - 1$,
  contradiction.
  
  The rest is proposition \ref{prop:easyPone} and
  the recurrence in theorem \ref{thm:vrecurrence}.
\end{proof}

This extra smoothness, of $\overline{X_f}$ at each $g$ covering $f$,
is known to hold for Schubert varieties (essentially from their normality).
However, this corollary was proven in the symplectic situation
\cite[theorem 1]{LittDH} without explicitly requiring this extra
smoothness, so perhaps it is automatic.

We describe this story (from \cite{LittDH}) in the case that $X$ is a
flag manifold, though to recapitulate it properly would involve
introducing a great deal of wholly standard notation, which we omit.
When $Y\subset Z$ are Schubert varieties
$\overline{X_{w r_\beta}} \subset \overline{X_w} \subseteq G/P$
projectively embedded in the $G$-representation $V_\lambda$,
the coefficient is %determined by the equation 
$v(Z)_{(f_0,f_1),Y} = (w r_\beta \cdot \lambda - w \cdot \lambda)/\beta$.
This is easily derived from the Chevalley-Monk rule for intersecting
a Schubert variety with a hyperplane, and is the basic step
in \cite{PS}.

\subsection{Linear relations among the $\{v_\gamma\}$}\label{ssec:linrels}

The Duistermaat-Heckman function is piecewise polynomial, as can be 
seen from either their formula or theorem \ref{thm:myDH}. However,
theorem \ref{thm:myDH} hugely overestimates the number of pieces --
it predicts a great many walls between regions of different polynomials 
that turn out to not actually be different. For example, in the case of
a toric variety, the D-H function is $1$ on the entire polytope, 
but theorem \ref{thm:myDH} breaks the polytope into a triangulation.

So anywhere within $\Phi_T(X)$ that we know for some other reason
there is {\em not} a jump in
the D-H function -- and we shall look nearby $\Phi_T(\min(X))$ --
we get a linear condition among the coefficients $\{v_\gamma\}$. 
While the connection may be obscured by the Fourier 
transform, the proposition following is essentially built on this idea.

\begin{Proposition}\label{prop:linrels}
  Assume $X$ has a unique supporting fixed point $\min(X)$.
  Let  $C_{\min(X)} X \subseteq %T_{\min(X)} X \leq 
  T_{\min(X)} X$ 
  denote the tangent cone to $X$ at $\min(X)$, 
  and the tangent space to $X$ at $\min(X)$, respectively. 
  The ($T$-invariant) tangent cone carries a Chow class 
  $[C_{\min(X)} X] \in A^*_T(T_{\min(X)} X) \iso \Sym(T^*)$.
  Denote by $[\{\vec 0\} \in T_{\min(X)} X] \in A^*_T(T_{\min(X)} X)$
%  $[T_{\min(X)} X \leq T_{\min(X)} \PP V] \in A^*_T(T_{\min(X)} \PP V)$ 
  the evident Chow class (a product of $T$-weights).

  Then
  $$ \sum_{\gamma} \frac{v_\gamma }
  {\prod_{f\in\gamma,\, f\neq \min(X)} 
  \left(\Phi_T(f) - \Phi_T(\min(X))\right) }
  = \frac
  {\left[C_{\min(X)} X \subseteq T_{\min(X)} X\right]} 
  {[ \{\vec 0\} \in T_{\min(X)} X]}
  \quad\text{as ratios in $\Sym(T^*)$}$$
  where the sum is over maximum-length closure chains in $X$. 
\end{Proposition}

\begin{proof}
  By shrinking $V$ to the linear span of $\hatX$ and invoking
  lemma \ref{lem:span}, we may assume that the $\Phi_T(\min(X))$-weight
  space in $V$ is $1$-dimensional. That lets us invoke
  proposition \ref{prop:congruence}, which says that 
  $$ [\hatX] / [0\times L] 
  \equiv [C_{\min(X)} X \subseteq T_{\min(X)} \PP V] 
  / [\{\vec 0\} \in T_{\min(X)} \PP V] $$
  as rational functions in $\Sym(T^*)[D]$ specialized at $D = -\Phi_T(f)$.

  Theorem \ref{thm:ChowDH} gives us a formula for the left side of
  this equation,
  which specializes at $D = -\Phi_T(f)$ to the left side of the
  desired equation.

  The right side is the equivariant multiplicity of $X$ at $\min(X)$,
  which can be computed inside either $T_{\min(X)} \PP V$ or $T_{\min(X)} X$.
\end{proof}

\begin{Corollary}\label{cor:linrels}
  Assume in addition that there exists a set $Q \subseteq X^T$ of fixed points
  such that the weights in $T_{\min(X)} X$, with repetition,
  are % multiples of
  $\{\Phi_T(q) - \Phi_T(\min(X)) : q\in Q\}$. 
  Then this formula can be rewritten inside $\Sym(T^*)$ as
  $$ \sum_\gamma v_\gamma 
  \prod_{f \in X^T\setminus \gamma} \left(\Phi_T(f) - \Phi_T(\min(X))\right)
  =  %c\, 
  [C_{\min(X)} X
 \subseteq T_{\min(X)} X] \prod_{f \in X^T \atop f\notin Q\cup \{\min(X)\}} 
  \left(\Phi_T(f) - \Phi_T(\min(X))\right). $$
%  where $c\in\rationals$ is one over the product of those multiples.
%  No term on the left side vanishes.

  Let $\check\sigma : T^* \to \integers$ be a linear functional,
  and write $\Phi_R := \check\sigma \circ \Phi_T$. 
  Assume that $\delta := \{f \in X^T : \Phi_R(f) = \Phi_R(\min(X))\}$
  is a closure chain, and that $\delta \not\subseteq Q\cup \{\min(X)\}$.
  Then
  $$ \sum_{\gamma \supseteq \delta} v_\gamma 
  \prod_{f\in  X^T\setminus\gamma} \left(\Phi_R(f) - \Phi_R(\min(X))\right)
  = 0 $$
  where the left side is a sum over maximum-length closure chains.
\end{Corollary}

\begin{proof}
  To get the first formula above, multiply
  both sides of the one from proposition \ref{prop:linrels} by 
  $\prod_{f\in X^T,\ f\neq \min(X)} \left(\Phi_T(f) - \Phi_T(\min(X))\right)$.

  The functional $\check\sigma$ induces a homomorphism
  $\Sym(T^*) \to \integers$, $\lambda \mapsto \rho(\lambda)$;
  applying it to the first formula we get the equation
  $$ \sum_\gamma v_\gamma 
  \prod_{f \in X^T\setminus \gamma} \left(\Phi_R(f) - \Phi_R(\min(X))\right)
  = %c\, 
  [C_{\min(X)} X \subseteq T_{\min(X)} X] 
  \prod_{f \in X^T \atop f\notin Q\cup \{\min(X)\}} 
  \left(\Phi_R(f) - \Phi_R(\min(X))\right). $$
  For any $\gamma\not\supseteq \delta$, one of the terms in
  the product
  $\prod_{f \in X^T\setminus \gamma} \left(\Phi_R(f) - \Phi_R(\min(X))\right)$
  is zero, so on the left side 
  it is enough to sum over $\gamma \supseteq \delta$.

  By the condition $\delta \not\subseteq Q\cup\{\min(X)\}$, 
  one of the terms in the right-hand product is zero.
\end{proof}

Some remarks:
\begin{itemize}
\item We used the notation
  $\Phi_R$ because the Pontrjagin dual of $\check\sigma$ is
  a homomorphism $R : \Gm \to T$, whose %associated
  moment map on the fixed points is this $\Phi_R$.
\item
  One can obtain many more such conditions by
  applying corollary \ref{cor:linrels} to components of
  $\overline{X_{f_0,\ldots,f_k}}$ of codimension $k$ in $X$, 
  and using theorem \ref{thm:vrecurrence}.  
\item If $X$ is irreducible, we can study instead its image under 
  the $\beta$ from proposition \ref{prop:branch} (picking up a
  factor from the degree of $\beta$ to its image).
  In this smaller projective space, it is easy to see that
  a set $Q$ as postulated in corollary \ref{cor:linrels} must exist,
  even for $T_{\min(X)} \PP V \geq T_{\min(X)} X$.
\item If $X$ has not only isolated fixed points but isolated fixed curves, as 
  in \cite{GKM}, then this set $Q$ exists canonically: take the $T$-fixed 
  points other than $\min(X)$ on the $T$-fixed curves passing through 
  $\min(X)$.  This condition holds for flag manifolds and toric varieties, 
  though not for the Bott-Samelson manifold from section \ref{sssec:BS}.
\end{itemize}

As usual, things are particularly simple for $X$ a toric variety,
where the simplicial complex $\Delta(X,S)$ is
a triangulation of the moment polytope $P$,
whose vertices correspond naturally to $X^T$.
The set $Q$ can (and must) be taken to be the vertices sharing an edge
with $\min(X)$. 
\junk{
  Now let $\delta$ be a closure chain with $1$ fewer than the
  maximum number of elements, a ``ridge'' in the complex $\Delta(X,S)$.
  By the fact that $\Delta(X,S)$ is homeomorphic to a ball, the face
  $\delta$ is contained in either one or two facets of $\Delta(X,S)$,
  depending on whether the ridge is exterior to $P$ or interior.
  In the case that it is interior, so $\delta \subset \gamma_1,\gamma_2$, 
  this equation says that the coefficients $v_{\gamma_1},v_{\gamma_2}$
  can be determined from one another. This is not a surprise, of course,
  since inside $P$ the D-H function has no jumps at all.
}

\subsection{Assembling the coefficients $\{v_\bargamma\}$}
\label{ssec:assembling}

The formula $\deg X = \sum_\gamma v_\gamma$ mentioned after
theorem \ref{thm:myDH}, summing over maximum-length closure chains, 
can be refined to $\deg X = \sum_\bargamma v_\bargamma$ summing
over maximum-length witnesses. We now give an inductive version
of this formula.

Given a closure chain $\gamma = (f_0,\ldots,f_k)$ and a component $Y$
of $\overline{X_{f_0,\ldots,f_k}}$ of dimension $\dim X - k$ (which 
requires $\gamma$ to be the initial segment of a maximum-length chain),
define
$$ v_{\gamma,Y} 
:= \sum_{\bargamma = (f_0 \in Y_0,\ldots,f_k \in Y_k = Y)} v(X)_{\bargamma} $$
where the summands were defined in section \ref{ssec:proofs}.
If $k=\dim X$, then $Y=\{f_{\dim X}\}$ by lemma \ref{lem:maxwitnesses},
and therefore $v_{\gamma,Y}$ is the $v_{\gamma}$ 
also defined in section \ref{ssec:proofs}.

\begin{Proposition}\label{prop:assembling}
  Fix $\gamma = (f_0,\ldots,f_k)$ and a component $Y$
  of $\overline{X_{f_0,\ldots,f_k}}$ of dimension $\dim X - k$.
  Pick a $T$-invariant hyperplane $\PP H$ not containing $f_k$, 
  and let $\{Z_i\}$ be the irreducible components of $Y_k \cap \PP H$.
  Then 
  $$ v_{(f_0,\ldots,f_k),Y} \deg Y 
  = \sum_i v_{(f_0,\ldots,f_k,\min(Z_i)), Z_i} \deg Z_i. $$
\end{Proposition}

\begin{proof}
  $$ \deg Y_k = \deg (Y_k \cap \PP H) = \sum_i m_i \deg(Z_i) $$
  where the $m_i$ are the multiplicities of the components $Z_i$ in
  the scheme $Y_k \cap \PP H$. The result follows by unwinding 
  the definition of $v(X)_{\bargamma}$.
\end{proof}

\section*{Acknowledgements}

This project has been gestating since shortly after \cite{LittDH}, and indeed 
this paper might be considered only a progress report toward \cite{AKfuture}.  
In that long period I have discussed it with many people. 
I best remember having useful discussions with Michel Brion (to whom I
am particularly grateful for remarks on previous versions of this
article), Rebecca Goldin, Francisco Santos, and Catalin Zara, and with
Mark Haiman and Bernd Sturmfels about the Hilbert scheme of $n$ points
in the plane.

\bibliographystyle{alpha}    % it seems this does nothing.

\begin{thebibliography}{10}

\bibitem[Ak81]{Akyildiz} E. Aky\i ld\i z,
  Bruhat decomposition via $G\sb{m}$-action, 
  Bull. Acad. Polon. Sci. Ser. Sci. Math. 28 (1980),
   no. 11-12, 541--547 (1981). 

\bibitem[AK]{branch} V. Alexeev, A. Knutson, 
  Complete moduli spaces of branchvarieties,
  preprint 2006. {\tt math.AG/0602626}

\bibitem[AB84]{AB} M. Atiyah, R. Bott,
  The moment map and equivariant cohomology,
  Topology 23 (1984) no. 1, 1--28.

\bibitem[BV84]{BV} N. Berline, M. Vergne,
  Un calcul de l'indice \'equivariant de l'op\'erateur de Dirac
  par la m\'ethode de la chaleur,
  C. R. Acad. Sci. Paris S\'er. I Math. 299 (1984), no. 11, 511--514.
  
\bibitem[BB76]{BB} A. Bia\l ynicki-Birula,
  Some properties of the decompositions of algebraic varieties
  determined by actions of a torus,
  Bull. Acad. Polon. Sci. Ser. Sci. Math. Astronom. Phys. 24 (1976),
  no.  9, 667--674.

\bibitem[BW82]{BjWa} A. Bj\"orner, M. Wachs,
  Bruhat order of Coxeter groups and shellability,
  Adv. in Math. 43 (1982), no. 1, 87--100. 

\junk{
\bibitem[Br91]{Brion} M. Brion, 
  Cohomologie \'equivariante des points semi-stables,
  J. Reine Agnew. Math. 421 (1991), 125--140.
}
\bibitem[BP90]{Brion} M. Brion, C. Procesi,
  Action d'un tore dans une vari\'et\'e projective,
  Operator algebras, Unitary Representations, 
  Enveloping Algebras, and Invariant Theory,
  Progress in Mathematics Vol. 92, (1990)
  Birkh\"a user, pp. 509--539.

\bibitem[Br97]{eqvtChow} M. Brion,
  Equivariant Chow groups for torus actions,
  Transformation Groups, Vol. 2, No. 3, 1997, pp. 225--267.

\bibitem[Br98]{eqvtHandA} \bysame,
  Equivariant cohomology and equivariant intersection theory,
  Notes de l'\'ecole d'\'et\'e "Th\'eories des repr\'esentations et
  g\'eom\'etrie alg\'ebrique" (Montr\'eal, 1997).
  {\tt math.AG/9802063}

\bibitem[DCP73]{wonderful} C. DeConcini, C. Procesi, 
  Complete symmetric varieties. Lecture Notes in Math. 996, 
  Springer, 1973, 1-­44.

\bibitem[Du03]{Duan} H. Duan,
  The degree of a Schubert variety,
  Adv. Math. 180 (2003), no. 1, 112--133.

\bibitem[DH82]{D-H} J. J. Duistermaat, G. Heckman,
  On the variation in the cohomology class of the symplectic form
  of the reduced phase space, 
  Invent. Math. 69 (1982), no. 2, 259--268.

\bibitem[GM06]{GM} R. Goldin, S. Martin,
  Cohomology pairings on the symplectic reduction of products,
  Canad. J. Math. 58 (2006), no. 2, 362--380. 
  {\tt math.SG/0408250}

\bibitem[GKM98]{GKM} M. Goresky, R. Kottwitz, R. MacPherson, 
  Equivariant cohomology, Koszul duality, and the localization theorem,  
  Invent. Math.  131  (1998),  no. 1, 25--83.

\bibitem[Gu94]{Guillemin} V. Guillemin, 
  Reduced phase spaces and Riemann-Roch,
  Lie theory and geometry, 305­334, Progr. Math., 123, 
  Birkh\"auser Boston, Boston, MA, 1994.

\bibitem[GLS88]{GLS} \bysame, E. Lerman, S. Sternberg, 
  On the Kostant multiplicity formula,
  J. Geom. Phys. 5 (1988), no. 4, 721--750.

\bibitem[GS95]{GScoeffs} \bysame, S. Sternberg, 
  The coefficients of the Duistermaat-Heckman polynomial,
  Geometry, Topology, and Physics for Raoul Bott,
  Conference proceedings and lecture notes in geometry and topology,
  Vol. IV, International Press, 1995, 202--213.

\bibitem[Ha02]{Ha} M. Haiman, 
  Combinatorics, symmetric functions and Hilbert schemes,
  Current Developments in Mathematics 2002, no. 1 (2002), 39--111. 

\bibitem[He81]{Hesselink} W. Hesselink, 
  Concentration under actions of algebraic groups,
  Paul Dubreil and Marie-Paule Malliavin Algebra Seminar, 33rd Year
  (Paris, 1980), pp. 55­89, Lecture Notes in Math., 867, Springer,
  Berlin, 1981.

\bibitem[Jo97]{Joseph} A. Joseph,
  Orbital varieties, Goldie rank polynomials and unitary
  highest weight modules,
  Algebraic and analytic methods in representation theory
  (S\o nderborg, 1994), 53--98.

\bibitem[Ju77]{Ju} J. Jurkiewicz,
   An example of algebraic torus action which determines the nonfiltrable
   decomposition,
   Bull. Acad. Polon. Sci. S\'er. Sci. Math. Astronom. Phys. 25
   (1977), no.  11, 1089--1092.

\bibitem[Ki86]{Kirwan} F. Kirwan, 
  Cohomology of quotients in symplectic and algebraic geometry,
  Mathematical Notes 31,
  Princeton University Press, Princeton, NJ, 1984.

\bibitem[Kn99]{LittDH} A. Knutson, 
  A Littelmann-type formula for Duistermaat-Heckman measures,
  Inventiones mathematicae, 135 (1999) no. 1, 185--200.

\bibitem[Kn06]{balanced} \bysame,
  Balanced normal cones and Fulton-MacPherson's intersection theory, \\
  {\sl Pure and Applied Mathematics Quarterly}
  Vol 2, \# 4, 2006 (MacPherson issue part II)

\bibitem[Kn]{AKfuture} \bysame, 
  Standard bases for homogeneous coordinate rings, in preparation.

\bibitem[KMY]{KMY} \bysame, E.~Miller and A.~Yong, 
  {Gr\"{o}bner geometry of vertex decompositions and of flagged tableaux}, 
  To appear in Crelle's journal,
  \texttt{math.AG/0502144}.

\bibitem[KZJ07]{KZJ} \bysame, P. Zinn-Justin,
  A scheme related to the Brauer loop model,
  Advances in Mathematics 214 (2007), Issue 1, 40--77.

\bibitem[Ko78]{Konarski} J. Konarski,
  Decompositions of normal algebraic varieties determined 
  by an action of a one-dimensional torus,
  Bull. Acad. Polon. Sci. S\'er. Sci. Math. Astronom. Phys. 26 (1978), 
  no. 4, 295­-300.

\bibitem[PS]{PS} A. Postnikov, R. Stanley,
  Chains in the Bruhat order, preprint. {\tt math.CO/0502363}

\bibitem[Ro89]{Rossmann} W.~Rossmann, 
  {Equivariant multiplicities on complex varieties},
  Orbites unipotentes et repr\'esenta\-tions, III, Ast\'erisque
  no.~173--174, (1989), 11, 313--330.
  
\end{thebibliography}

\end{document}